\documentclass[10pt]{amsart}
\usepackage{latexsym,amsmath,amssymb,graphicx}
\usepackage{float}
\usepackage[bf, small]{caption}
\usepackage[all,web]{xy}
\usepackage{color}
\usepackage{hyperref}

\usepackage{xcolor}

\def\xyellowspace{%
  \sbox0{\colorbox{yellow}{\strut\ }}
  \dimen0=\wd0\relax
  \hskip0pt\cleaders\box0\hskip\dimen0\hskip0pt}

\begingroup
\catcode`\ =\active%
\gdef\makeyellowspace{\let \xyellowspace\catcode`\ =\active}%
\endgroup

\def\?#1{\colorbox{yellow}{\strut#1}}

\DeclareFontFamily{OT1}{rsfs10}{}
\DeclareFontShape{OT1}{rsfs10}{m}{n}{ <-> rsfs10 }{}
\DeclareMathAlphabet{\mathscript}{OT1}{rsfs10}{m}{n}

\DeclareMathOperator{\im}{Im}       
\DeclareMathOperator{\Spec}{Spec}   
\DeclareMathOperator{\Hom}{Hom}     
\DeclareMathOperator{\Tors}{Tors}    
\DeclareMathOperator{\Pic}{Pic}     
\DeclareMathOperator{\Cl}{Cl}       
\DeclareMathOperator{\Div}{Div}     
\DeclareMathOperator{\Cox}{Cox}     
\DeclareMathOperator{\Aut}{Aut}     
\DeclareMathOperator{\rk}{rk}       
\DeclareMathOperator{\Sing}{Sing}   
\DeclareMathOperator{\Br}{Br}       
\DeclareMathOperator{\Nef}{Nef}     
\DeclareMathOperator{\codim}{codim} 
\DeclareMathOperator{\Char}{char}   
\DeclareMathOperator{\conv}{Conv}   

\title[1-coverings and Mori Dream Spaces]{\'{E}tale coverings in codimension 1 with applications to Mori Dream
Spaces}

\author[M. Rossi]{Michele Rossi}

\date{\today}

\address{Dipartimento di Matematica e Applicazioni, Universit\`a di Milano-Bicocca,\newline
Ed.~U5-Ratio, via Roberto Cozzi, 55, 20125, Milano} \email{michele.rossi@unimib.it}

\thanks{The author was partially supported by the MIUR-PRIN 2010-11 Research Funds ``Geometria delle Variet\`{a}
Algebriche'' and MIUR-PRIN 2022 Research Funds ``Moduli spaces and birational geometry''. He is also supported by the I.N.D.A.M. as a member of the G.N.S.A.G.A.}

\def \wrt{with respect to }

\def \a{\alpha }
\def \b{\beta }
\def \d{\delta }

\def \Ga{\Gamma }
\def \Si{\Sigma }
\def \g{\gamma}
\def \vf{\varphi}

\def \ét{\'{e}tale}

\def \q{\mathbf{q}}
\def \pp{\mathbf{p}}

\def \v{\mathbf{v}}

\def \x{\mathbf{x}}

\def \1{\mathbf{1}}
\def \0{\mathbf{0}}

\def\P{{\mathbb{P}}}

\def\p2{\mathbb{P}^2}
\def\p3{\mathbb{P}^3}
\def\p4{\mathbb{P}^4}

\def\cO{\mathcal{O}}

\def\rk{\operatorname{rk}}

\def\GL{\operatorname{GL}}

\def\Z{\mathbb{Z}}

\def\C{\mathbb{C}}
\def\K{\mathbb{K}}
\def\R{\mathbb{R}}

\def\Q{\mathbb{Q}}
\def\N{\mathbb{N}}
\def\T{\mathbb{T}}

\def\X{\mathfrak{X}}
\def\cox{\mathcal{C}ox}
\def\irr{\mathcal{I}rr}
\def\cS{\mathcal{S}}

\def\SF{\mathcal{SF}}

\def\G{\mathcal{G}}

\def\I{\mathcal{I}}
\def\Ls{\mathcal{L}}
\def\Ga{\Gamma}
\def\De{\Delta}

\def\sqm{s\,$\Q$m }

\def\pet{\pi_1^{\text{\rm{\'{e}t}}}}
\def\Bmu{\boldsymbol\mu}
\def\U1{\mathfrak{U}^{(1)}}

\theoremstyle{plain}
\newtheorem{theorem}{Theorem}[section]
\newtheorem{proposition}[theorem]{Proposition}

\newtheorem{thm-def}[theorem]{Theorem--Definition}
\newtheorem{corollary}[theorem]{Corollary}

\newtheorem{lemma}[theorem]{Lemma}

\newtheorem*{a-proposition}{Proposition}

\theoremstyle{remark}
\newtheorem{remark}[theorem]{Remark}
\newtheorem{remarks}[theorem]{Remarks}

\newtheorem*{notation}{Notation}

\theoremstyle{definition}
\newtheorem{definition}[theorem]{Definition}

\newtheorem*{step I}{Step I}
\newtheorem*{step II}{Step II}
\newtheorem*{step III}{Step III}
\newtheorem*{step IV}{Step IV}

\newtheorem*{acknowledgements}{Acknowledgements}

\newcommand{\halfline}{\vskip6pt}
\begin{document}


\begin{abstract} The present paper is devoted to developing relations between Galois \ét coverings in codimension 1
and \ét fundamental groups in codimension 1 of algebraic varieties, aimed to studying the topology of Mori dream
spaces. In particular, the universal \ét covering in codimension 1 of a non-degenerate toric variety and a canonical
Galois \ét covering in codimension 1 of a Mori dream space (MDS) are exhibited. Sufficient conditions for the latter
being either still a MDS or the universal \ét covering in codimension 1 are given. As an application, a canonical
toric embedding of K3 universal coverings, of Enriques surfaces which are Mori dream, is described.
\end{abstract}
\keywords{Galois \ét covering, \ét covering in codimension 1, \ét fundamental group, non-degenerate toric variety,
Gale duality, fan matrix, weight matrix, small $\Q$-factorila modification, Mori dream space, Minimal Model Program,
weak Lefschetz theorem, Enriques surface, K3 surface.}
\subjclass[2010]{ 14H30 \and 14E20\and 14M25 \and 14E30}

\maketitle

\tableofcontents

\section{Introduction}

The main topics of the present paper are \emph{\ét coverings in codimension 1} between algebraic varieties, in the
following simply called \emph{1-coverings}, aimed to study the topology of Mori dream spaces (MDS). A 1-covering is
a finite morphism, \ét over a Zariski open subset of the domain, whose complementary closed subset has codimension
strictly greater than 1 (see Definition~\ref{def:1-covering}).

 1-coverings were studied  in some detail by F.~Catanese in \cite{Catanese}, although in the slightly broader sense of
 \emph{quasi-\ét} morphisms, i.e. \emph{quasi-finite} morphisms, \ét in codimension 1. More recently, there was a
 renewed interest about this topic in relation with the Koll\`{a}r conjecture asserting that the local fundamental
 group (that is the fundamental group of the link) of a log terminal singularity should be finite \cite[Question
 26]{Kollar}. This fact motivated a number of very interesting results about finiteness condition of (local)
 fundamental groups of algebraic varieties and relations between the fundamental group of the regular locus and the
 global one, both over the complex field and in positive characteristic: see e.g.
 \cite{BC-RGST},\cite{C-RST},\cite{GKP},\cite{Stibitz},\cite{Tian-Xu},\cite{Xu}. At this purpose, notice that, in the
 very recent preprint \cite{Braun}, L.~Braun gives a proof of the Koll\`{a}r conjecture.

  In this context, the study of 1-coverings and related (\ét) fundamental groups is motivated by giving an algebraic
  proof of W.~Buczy{\'n}ska's results, appeared in 2008 in a still unpublished paper \cite{Buczynska}, to extending to MDS
  some results previously obtained for $\Q$-factorial, complete toric varieties, in the paper \cite{RT-QUOT}, jointly
  written with L.~Terracini.

Buczy{\'n}ska's approach is firstly resumed, by revising her topological results in \cite{Buczynska} from the
algebraic-\ét point of view. In particular, the \emph{\ét fundamentale group in codimension 1} is introduced (see
Definition~\ref{def:etalep11}) as the algebraic reformulation of the same topological notion given by
\cite[Def.~3.1]{Buczynska}: namely, the former is the \emph{pro-finite completion} of the latter. Then, what has been
here obtained about relations between 1-coverings and the algebraic fundamental group in codimension 1 is holding on a
general algebraically closed field $\K=\overline{\K}$, with $\Char \K=0$. This is the content of \S~\ref{ssez:pi11}
and \S~\ref{ssez:Galois 1-cov}: the notion of the \ét fundamental group in codimension 1 looks to be a new one in the
literature, at least as far as the author knows. Then the theory here developed seems to be an original one, although
essentially analogous to the theory of the global \ét fundamental group, quickly recalled in \S~\ref{ssez:etalepi1}.
As observed in Remarks~\ref{rem:K=C+finite} and \ref{rem:K=C+normal} results here obtained, like e.g.
Theorem~\ref{thm:excisione}, Corollary~\ref{cor:pi1smooth} and Theorem~\ref{thm:pi1normale}, do not imply their
analytical analogous statements proved by Buczy{\'n}ska in \cite{Buczynska}, unless the involved fundamental groups are
finite, as in the important case of toric varieties, but probably of more general MDS after \cite{BCHMcK}, \cite{GOST}
and a very recent Braun result proving that the fundamental group of a weak Fano variety is finite \cite{Braun} (see
consideration ending up Remark~\ref{rem:K=C+finite}).

Consequently \S~\ref{sez:toric} is devoted to apply results of previous sections to toric varieties, so obtaining a
natural field extension of results proved in \cite[\S~4]{Buczynska}. In particular, Theorem~\ref{thm:QUOT+} shows that
a non-degenerate toric variety always admits a universal 1-covering, which is still a non-degenerate toric variety:
this is an extension of \cite[Thm.~2.2]{RT-QUOT} in which the same statement was proved for a complex, complete and
$\Q$-factorial toric variety. Let me here recall that, as for the universal covering, in general, an algebraic variety
does not admit a universal 1-covering. Then the main interest of Theorem~\ref{thm:QUOT+} resides in defining a class
of algebraic varieties, namely non-degenerate toric varieties, giving an exception toward such a general fact.

Recalling that a MDS has a canonical toric embedding, what proved in \S~\ref{sez:toric} applies to give interesting
consequences on the topology of a MDS. This is the content of \S~\ref{sez:MDS}, where we considered a slightly broader
(\wrt MDS) category of spaces called, coherently with \cite{R-wMDS}, \emph{weak} Mori dream spaces (wMDS). A wMDS
admitting a projective closed embedding is a MDS in the usual Hu-Keel sense \cite{Hu-Keel}. Probably the main result
here obtained is the construction of a canonical 1-covering $\widetilde{X}$ of a wMDS $X$, given by
Theorem~\ref{thm:notors1-cov}. In particular, such a canonical 1-covering comes with a canonical closed embedding into
the universal 1-covering $\widetilde{W}$ of the the canonical ambient toric variety $W$ of $X$, whose existence is
guaranteed by Theorem~\ref{thm:QUOT+}. Unfortunately, this canonical embedding between 1-coverings does
not turn out to be a \emph{neat} embedding (see Def.~\ref{def:neat}), in general: but the latter is shown to be
equivalent with the condition that the 1-covering $\widetilde{X}$ is still a wMDS whose Cox ring is isomorphic to the Cox ring of $X$ (see Theorem~\ref{thm:notors1-cov}).

The following \S~\ref{ssez:neat} and \S~\ref{ssez:gruppifondamentali} are dedicated to studying properties of the
canonical embedding $\widetilde{X}\hookrightarrow\widetilde{W}$ and the topology of $\widetilde{X}$ itself,
respectively. In particular, as a consequence of results of M.~Artebani and A.~Laface~\cite{AL}, S.-Y.~Jow \cite{Jow}
and G.~Ravindra and V.~Srinivas \cite{RS}, Proposition~\ref{prop:s.c.neat} gives some sufficient conditions for
$\widetilde{X}\hookrightarrow\widetilde{W}$ being a neat embedding, hence the canonical 1-covering $\widetilde{X}$
still being a wMDS. On the other hand, by applying deep results of M.~Goresky and R.~Mac~Pherson \cite{G-MacPh},
Theorem~\ref{thm:WeakLefschetz} gives a sufficient condition for the canonical 1-covering
$\widetilde{X}\longrightarrow X$ being the universal one, in the complex case $\K=\C$.

The present paper is organized as follows. \S~\ref{ssez:etalepi1} is dedicated to quickly recall standard facts on \ét
coverings and \ét fundamental groups and to proving Excision Theorem~\ref{thm:excisione}: it gives an algebraic-\ét
counterpart of \cite[Thm.~3.4]{Buczynska} (see Remark~\ref{rem:K=C+finite}). \S~\ref{ssez:esistenza} is devoted to
recalling relations between the \ét fundamental group and the universal covering, when existing, of an algebraic variety.
The following \S~\ref{ssez:pi11} and \S~\ref{ssez:Galois 1-cov} introduce the \ét fundamental group in codimension 1
and local Galois 1-coverings: these are essentially new topics. Let me underline that, in this context the adjective
\emph{local} is associated with Galois 1-covering and not to a concept of fundamental group, so avoiding any confusion
with the concept of local fundamental group, recently studied in connection with Kollar conjecture, as already
mentioned above, and not treated in the present paper.  Main result of this section is Theorem~\ref{thm:pi1normale},
relating the \ét fundamental group in codimension 1 of a normal variety with the \ét fundamental group of its regular
locus, so giving an algebraic-\ét counterpart of \cite[Cor.~3.10]{Buczynska}. Then \S~\ref{ssez:pullback} ends up
\S~\ref{sez:1-covering} by fixing notation on divisors' pull back. As
already described above, \S~\ref{sez:toric} and \S~\ref{sez:MDS} are devoted to applying results and techniques,
developed in \S~\ref{sez:1-covering}, to toric varieties and wMDS, respectively. The last \S~\ref{sez:esempi} gives
evidences of both positive and negative occurrences in Theorems~\ref{thm:notors1-cov} and \ref{thm:WeakLefschetz}, by means of three interesting
examples. The first example is given by Example~\ref{ssez:HK}, describing a case in which the canonical 1-covering is still a
wMDS (actually a MDS): this example was borrowed from id no. 97 in \cite{CRdb}. The second example is a new one, as far as the author's knowledge allows to conclude: it describes a MDS whose canonical 1-covering is the universal 1-covering admitting a neat embedding into the universal 1-covering of the canonical ambient toric variety. The third example is given by very special
families of Enriques surfaces which are Mori dream spaces. Their canonical 1-covering is also their universal \ét
covering, hence a K3 surface which can never be a MDS, as admitting an infinite automorphism group. In this case
Theorem~\ref{thm:notors1-cov} gives interesting information about this kind of special Enriques surfaces, their K3
universal coverings and the associated canonical toric embeddings (see Cor.~\ref{cor:Enriques} and Rem.~\ref{rem:}).

Main original contributions of the present paper are then given by:
\begin{itemize}
  \item the theory of the \ét fundamental group in codimension 1 and local Galois 1-coverings, developed in \S~1.3
      and \S~1.4, giving the algebraic-\ét counterpart of Buckcinska's results provided in \cite{Buczynska};
  \item Theorem~\ref{thm:QUOT+} extending the main result (Thm.~2.2) of \cite{RT-QUOT} from complex, $\Q$-factorial,
      complete toric varieties to a more general non-degenerate toric variety over $\K$;
  \item Theorem~\ref{thm:notors1-cov} providing an analogue of the previous result in the broader context of MDS: in
      particular \S~\ref{ssez:neat} and \S~\ref{ssez:gruppifondamentali} study conditions to getting a universal
      1-covering of a MDS with a neat canonical toric embedding.
\end{itemize}

\section{\'{E}tale covering in codimension 1 (1-covering)}\label{sez:1-covering}

The present section is devoted to recalling and extend to any algebraically closed field $\K$, with $\Char \K=0$,
concepts and results introduced in \cite[\S~3]{Buczynska}, under the assumption $\K=\C$. Notice that results here
given cannot in general replace Buczy{\'n}ska's results in \cite{Buczynska} about the \emph{fundamental group in
codimension 1} of a complex algebraic variety, since known conditions on the pro-finite completion $\widehat{G}$ of a
group $G$ do not transfer to the group $G$ itself, except for the particular case $G$ \emph{finite}.

\halfline
\emph{Notation}. Throughout the present paper:
 \begin{itemize}
   \item an algebraic variety is an integral scheme of finite type over an algebraically closed field $\K$ \cite[Prop.~II.4.10]{Hartshorne}; then it is always implicitly irreducible and reduced; the field $\K$ is always assumed with $\Char\K=0$;
   \item a \emph{small} closed subset $C$ of an algebraic variety $X$ is a
Zariski closed $C\subset X$ such that $\codim_X C> 1$; the complementary set $X\setminus C$ is called a \emph{big}
open subset of $X$;
   \item a morphism of algebraic varieties $\phi:Y\longrightarrow X$, is called an \emph{\ét
covering} if it is a \emph{finite \ét morphism} \cite[Def.~5.2.1]{Szamuely}; since $X$ is irreducible then $\phi$ is surjective with finite fibres
of constant cardinality called the \emph{degree} of $\phi$ ($\deg\phi$).
 \end{itemize}
The following is the key definition of the present paper: what is meant by \emph{\ét covering in codimension 1} of an
algebraic variety $X$.

\begin{definition}[1-covering]\label{def:1-covering}
  Let $\phi:Y\longrightarrow X$ be a morphism of algebraic varieties over $\K$. Then $\phi$ is called an
  \emph{\'{e}tale covering in codimension 1} (or simply a \emph{1-covering}) if it is finite and \'{e}tale in
  codimension 1, that is, there exists a \emph{small} Zariski closed subset $C\subset X$ such that
  $$\xymatrix{\phi|_{Y_C}:Y_C:=\phi^{-1}(X\setminus C)\ar[r]&X\setminus C}$$
  is a finite and \'{e}tale morphism onto the the complementary \emph{big} open subset $X\setminus C$. The smallest (with respect to inclusion) small
  closed $C$ satisfying this condition is called the \emph{branching locus of $\phi$} and denoted by $C=\Br \phi$.\\
  The degree of the \ét covering $\phi|_{Y_C}$ is called the  \emph{degree of the 1-covering $\phi$}, that is
  $\deg\phi:=\deg(\phi|_{Y_C})$.\\
  Recall that the automorphism group $\Aut(\phi)$ of an \ét covering $\phi:Y\longrightarrow X$ is the group of
  isomorphisms $\vf:Y\longrightarrow Y$ such that $\phi=\phi\circ\vf$. A \emph{connected} finite \ét covering is called \emph{Galois} if
  $|\Aut(\phi)|=\deg\phi$. By the following Proposition~\ref{prop:lifting} this is the same of asking that
  $\Aut(\phi)$ acts transitively over the fibres.\\
  A \emph{Galois} 1-covering is a 1-covering $\phi:Y\longrightarrow X$ such that $|\Aut(\phi|_{Y_C})|=\deg\phi$, where
  $C=\Br\phi$. This means that $\Aut(\phi|_{Y_C})$ acts transitively over the fibres of points in $X\setminus C$. In
  the following we will denote $$\Aut^{(1)}(\phi):=\Aut(\phi|_{Y_C})$$
\end{definition}

Recall the following fundamental result, as it will be very useful in the following.

\begin{theorem}[Zariski-Nagata Purity Theorem - See e.g. Thm.~5.2.13 in \cite{Szamuely}]\label{thm:ZNpurity}
  Let $\phi:Y\longrightarrow X$ be a finite surjective morphism of integral schemes with $Y$ normal and $X$ regular. Then the closed subscheme $\Br \phi\subset X$, over which $\phi$ is not \ét, must be of pure codimension 1.
\end{theorem}

\begin{corollary}\label{cor:ZNpurity}
  Every 1-covering $\phi:Y\longrightarrow X$ with $Y$ normal, restricts to give an \ét covering of the regular locus $X_{\text{\rm reg}}$.
\end{corollary}

\subsection{The \'{e}tale fundamental group of an algebraic variety}\label{ssez:etalepi1}
Recall that the \emph{\'{e}tale} (or \emph{algebraic}) \emph{fundamental group} of an algebraic variety $X$,
with a chosen \emph{base point} given by a closed point $x\in X$, is defined as the automorphism group of the fiber functor $F^x$ assigning to each finite
\ét covering $\phi:Y\longrightarrow X$ the finite set given by its fibre $F^x(\phi):=\phi^{-1}(x)$ over the base point
$x$ (see e.g. \cite[Def.~5.4.1]{Szamuely}). Then the \ét fundamental group is a functor from the category of finite \ét
coverings to the category of groups. Grothendieck proved that it is pro-representable \cite{SGA1},
\cite[Prop.~5.4.6]{Szamuely}, that is it can be represented as the inverse limit
\begin{equation*}
  \pet(X,x):=\varprojlim_{i\in \mathfrak{I}}\Aut(\phi_i)
\end{equation*}
running through all the Galois coverings $\{X_i\stackrel{\phi_i}\to X\}_{i\in \mathfrak{I}}$.

  Recall the following key fact about \ét morphisms:

  \begin{proposition}[\cite{Milne}, Cor.~2.16\,;\ \cite{Szamuely}, Cor.~5.3.3]\label{prop:lifting}
    Let $\phi:Y\longrightarrow X$ be a finite \ét covering of an algebraic variety $X$ and $f:Z\longrightarrow X$ be a morphism from a connected scheme $Z$. Let $\vf,\vf':Z\longrightarrow Y$ be morphisms lifting
    $f$, that is such that
    \begin{equation*}
    \phi\circ\vf=f=\phi\circ\vf'\ :\quad  \xymatrix{&&Y\ar[d]^-\phi\\
                Z\ar@<-.5ex>[urr]_-{\vf'}\ar@<.5ex>^-{\vf}[urr]\ar[rr]_-f&&X}
    \end{equation*}
    If there exists $z\in Z$ such that $\vf(z)=\vf'(z)$ then $\vf=\vf'$.
  \end{proposition}

   A first consequence of Proposition~\ref{prop:lifting}, is that the transitive action of the Galois group
   $\Aut(\phi_i)$ can be represented by acting on $\phi_i^{-1}(x)$ with a subgroup of the group $\mathfrak{S}_i^x$ of
   cyclic permutations. In fact, every non-trivial automorphism of the representing fibre $\phi_i^{-1}(x)$ cannot fix
   any point.

\begin{proposition}[\cite{Szamuely}, Cor.~5.5.2]\label{prop:puntobase}
  For any $x,x'\in X$ there exists an isomorphism
  $$\pet(X,x)\cong\pet(X,x')$$
  well defined up to conjugation.
\end{proposition}

\begin{proposition}[\cite{Murre}, Chap.~V and \cite{Szamuely}, \S~5.5, pg.~178]\label{prop:f*}
  Let $f:(Y,y)\longrightarrow (X,x)$ be a morphism of pointed algebraic varieties, that is  $x=f(y)$. Then
  there exists an induced homomorphism of \ét fundamental groups:
  \begin{equation*}
    \xymatrix{f_*:\pet(Y,y)\ar[r]&\pet(X,x)}
  \end{equation*}
\end{proposition}

\begin{remark}\label{rem:K=C}
  For $\K=\C$ the Riemann Existence Theorem \cite[Thm.~XII.5.1]{SGA1} gives a canonical isomorphism between the \ét
  fundamental group $\pet(X,x)$ and the \emph{pro-finite completion} of the fundamental group
  $\pi_1(X^{\text{an}},x)$, that is
  \begin{equation*}
    \pet(X,x)\cong \widehat{\pi}_1(X^{\text{an}},x):= \varprojlim_{N\lhd
    \pi_1(X^{\text{an}},x)}\left(\left.\pi_1(X^{\text{an}},x)\ \right/\ N\right)
  \end{equation*}
  where $N$ ranges through all the normal subgroups with finite index of $\pi_1(X^{\text{an}},x)$
  \cite[Cor.~5.2]{SGA1}. Notice that $\pi_1(X^{\text{an}},x)$ naturally maps onto each of its quotients, giving rise
  to a canonical map $\pi:\pi_1(X^{\text{an}},x)\longrightarrow\widehat{\pi}_1(X^{\text{an}},x)$.
  If $\pi_1(X^{\text{an}},x)$ is a \emph{finite group} then $\pi$ is an \emph{isomorphism}.

    Propositions~\ref{prop:lifting}, \ref{prop:puntobase} and \ref{prop:f*} are  generalizations, to
    every algebraic closed field $\K$ with $\Char\K=0$, of well known topological analogous results. In particular,
    for $\K=\C$, Prop.~\ref{prop:puntobase} can be obtained as an immediate consequence, passing to pro-finite
    completions, of the isomorphism $\pi_1(X^{\text{an}},x)\cong\pi_1(X^{\text{an}},x')$ obtained by choosing a path
    connecting $x$ and $x'$.
\end{remark}

In a sense, the following result, which is of fundamental importance for what follows, reverses Corollary~\ref{cor:ZNpurity} of the Zariski–Nagata Purity Theorem~\ref{thm:ZNpurity}.

\begin{lemma}\label{lem:estensione}
Let $X$ be a normal algebraic variety and $X_0\subseteq X$ be a big open subset.
Then, a finite \ét
covering of $X_0$
$$\phi:U\longrightarrow X_0$$
can be extended to give  a 1-covering
$$\overline{\phi}:\overline{U}\longrightarrow X$$
with $\overline{U}$ a normal algebraic variety and $\Br(\overline{\phi})\subseteq C:=X\setminus X_0$.
\end{lemma}

\begin{proof}
    Let $\Spec(A)=:V\subseteq X$ be an open Zariski subset and set $V_0=V\cap X_0$. Then $C\cap V=V\setminus V_0$ is a Zariski closed subset of codimension $\ge 2$ in $V$, that is $C\cap V=\mathcal{V}(I)$ with $I=(f_1,\ldots, f_c)$ a suitable ideal of $A$. Then
    \begin{equation*}
      A_i:=A_{f_i}\ ,\quad V_{i}:=\Spec(A_{i})\ :\quad V_0=\bigcup_{i=1}^c V_{i}
    \end{equation*}
    Both $A$ and $A_i$ are domains so that their fields of fractions are well defined. Since $V_i$ is a Zariski open of $V$, those fields of fractions turn out to be isomorphic, that is
    $$K:=K(A)\cong K(A_i)\cong \cO_\xi$$
     where $\cO_\xi$ is the local ring of a generic point $\xi\in V_i\subseteq V$ (see e.g. \cite[Ex.~II.3.6]{Hartshorne}). Consider $U_i=\phi^{-1}(V_i)$ and the finite étale covering
     $$\phi_i=\phi|_{U_i}:U_i\longrightarrow V_i$$
     Then $U_i=\Spec(B_i)$ where $B_i$ is a finitely generated $A_i$-module and a domain, so that $L_i:=K(B_i)$ is well defined and turn out to be a finite and (and separable, as $\Char K=0$) field extension of $K(A_i)\cong K$. Let $\widetilde{L}$ be the smallest subfield of the algebraic closure $\overline{K}$ containing every $L_i$. Then $\widetilde{L}$ is still a finite field extension of $K$: infact, every $L_i$ is a finite (hence algebraic) extension of $K$ (see e.g. \cite[Prop.~V.1.1]{Lang}) so that $\widetilde{L}$ is a $K$-vector subspace of a finitely generated field extension $K(\a_1,\cdots,\a_N)$ of $K$, with $\a_i$ algebraic over $K$, and then a finite extension of $K$ \cite[Prop.~V.1.6]{Lang}.
     Notice that there are canonical ring morphisms
     \begin{equation*}
       A\hookrightarrow A_i\hookrightarrow B_i\hookrightarrow L_i\hookrightarrow\widetilde{L}
     \end{equation*}
    so that one can consider the integral closure $\widetilde{B}$ of $A$ in the field $\widetilde{L}$ and the induced natural map $A\to \widetilde{B}$ giving rise to a map of affine schemes
    \begin{equation*}
      \widetilde{\phi}:\widetilde{U}:=\Spec(\widetilde{B})\longrightarrow \Spec(A)=V
    \end{equation*}

    \textbf{Finiteness of} $\widetilde{\phi}$. By Prop.~5.17 in \cite{A-McD}, $\widetilde{B}$ is a submodule of a free finitely generated $A$-module. Since $A$ is a Noetherian ring then $\widetilde{B}$ turns out to be a finitely generated $A$-module (see e.g. Propositions~6.2,~6.5 in \cite{A-McD}).

    \textbf{Local compatibility with} $\phi_i$. For any $i$ consider the tensor product $$\widetilde{B}\otimes_A A_i$$
    There is a natural morphism $\widetilde{B}\otimes_A A_i\longrightarrow B_i$
    giving rise to the following commutative diagram of ring morphisms
    \begin{equation*}
      \xymatrix{A_i\ar@{^(->}[r]&\widetilde{B}\otimes A_i\ar[r]& B_i\\
                A\ar@{^(->}[u]\ar@{^(->}[r]&\widetilde{B}\ar@{^(->}[u]\ar[ur]&}
    \end{equation*}
    and therefore the associated commutative diagram of schemes morphisms
    \begin{equation*}
      \xymatrix{V_i\ar@{^(->}[d]&\widetilde{U}\times_V V_i\ar[d]\ar[l]& U_i\ar@/_2pc/[ll]_-{\phi_i}\ar[l]\ar[dl]\\
                V&\widetilde{U}\ar[l]_-{\widetilde{\phi}}&}
    \end{equation*}
    so that, by construction, $\phi_i$ turns out to be the localization over $V_i$ of $\widetilde{\phi}$.

    \textbf{Ramification of} $\widetilde{\phi}$. By construction $\widetilde{\phi}$ is finite and \ét over $V_0$ but it can ramify over the complement $V\setminus V_0=C\cap V$, that is a small closed subset of $V$. Then, $\Br(\widetilde{\phi})$ has codimension $\ge 2$ in $V$, so that $\widetilde{\phi}$  is a 1-covering.

    \textbf{Normality of} $\widetilde{U}$. Notice that $\widetilde{B}$ is integrally closed by construction, as it is a finitely generated $A$-module and the integral closure of $A$ in $\widetilde{L}$.

     \textbf{Globalization and construction of} $\overline{U}$. To definitely prove the statement, consider an affine open covering
     $$\{V_j=\Spec (A_j)\}_{j=1}^s$$
     of $X$ and repeat the previous construction for any $j$, so getting $s$ 1-coverings $$\widetilde{\phi}_j:\widetilde{U}_j\longrightarrow V$$
     Following the lines of [Iitaka \S~2.14] one gets a global 1-covering
     $$\overline{\phi}:\overline{U}=\bigcup_j\widetilde{U}_j\longrightarrow \bigcup_j V_j=X$$
     extending the \ét covering $\phi:U\longrightarrow X_0$.
  \end{proof}

\begin{remark}\label{rem:estensione}
 Consider the case $\K=\C$ and let $X$ be a smooth complex algebraic variety. Let $X^{\text{an}}$ be the
 corresponding complex manifold endowed with the analytic topology, \wrt which $X^{\text{an}}$ turns out to be
 path-connected and semi-locally simply connected. Then the Riemann Existence Theorem \cite[Thm.~XII.5.1]{SGA1}
 establishes a categorical equivalence between the category of finite \ét coverings of $X$ and the category of finite
 topological coverings of $X^{\text{an}}$ \cite{GR},\cite[thm.~3.4]{Milne}. In particular, this implies that the
 analytic counterpart of Definition~\ref{def:1-covering} is \cite[Def.~3.13]{Buczynska}. Then Lemma~\ref{lem:estensione} implies and improves \cite[Lemma~3.15]{Buczynska}.
\end{remark}

\begin{theorem}[compare with Cor.~5.2.14 in \cite{Szamuely}]\label{thm:excisione}
  Let $X$ be a smooth  algebraic variety. Then there is an equivalence between the category of finite \ét coverings of $X$ and the category of finite \ét coverings of a big open subset $X_0\subseteq X$.

  In particular, for any point
  $x\in X_0$, the inclusion $i:X_0\hookrightarrow X$ induces an isomorphism
  $$\xymatrix{i_*:\pet(X_0,x)\ar[r]^-\cong&\pet(X,x)}$$
\end{theorem}

\begin{proof} Being $Y\twoheadrightarrow X$ an \ét covering of $X$, the functor giving the equivalence is constracted by sending $Y\mapsto Y\times_{X} X_0$.
Its fully faithfulness is a direct consequence of Lemma~\ref{lem:estensione} and Corollary~\ref{cor:ZNpurity} of the Zariski-Nagata purity Theorem~\ref{thm:ZNpurity}.

Finally the fact that $i_*$ gives an isomorphism follows by the definition of \ét fundamental groups as automorphism groups of fiber functors.
\end{proof}

\begin{remark}\label{rem:K=C+finite}
  In \cite[Thm. 3.4]{Buczynska} Buczy{\'n}ska proved a statement which is the analogue of Theorem \ref{thm:excisione} in
  the particular case $\K=\C$ and for the fundamental group $\pi_1(X^{\text{an}},x)$, under the further hypothesis
  that $C:=X\setminus X_0$ is also smooth: in fact her proof is essentially based on differential-topological techniques. In the
  Appendix of \cite{Buczynska} she sketched a road map to dropping such a smoothness condition on $C$.

  Notice that, if $\K=\C$ then Theorem~\ref{thm:excisione} does not imply in general \cite[Thm. 3.4]{Buczynska},
  unless the fundamental group $\pi_1(X^{\text{an}}_0,x)$ is finite: in this case the Buczy{\'n}ska's
  result is obtained without any smoothness assumption on the complement $C$. In fact, in this case $$\pi_1(X^{\text{an}}_0,x)\cong
  \widehat{\pi}_1(X^{\text{an}}_0,x)\cong\pet(X_0,x)\stackrel{\text{Thm.~\ref{thm:excisione}}}\cong\pet(X,x)$$
  On the other hand, since $X$ is smooth (hence normal) the inclusion
  $X_0\hookrightarrow X$ induces a surjection
  $$\pi_1(X^{\text{an}}_0,x)\twoheadrightarrow \pi_1(X^{\text{an}},x)$$
  (see e.g. \cite[Thm.\,12.1.5]{CLS}) so that also $\pi_1(X^{\text{an}},x)$ is finite and
  \begin{equation*}
    \pi_1(X^{\text{an}},x)\cong
  \widehat{\pi}_1(X^{\text{an}},x)\cong\pet(X,x)
  \end{equation*}

A few words about those \emph{finiteness hypotheses} on the analytical fundamental groups. It is a well known fact that the
fundamental group of a non-degenerate toric variety is finite (see \cite[Thm.~12.1.10]{CLS} and considerations opening
\S~\ref{ssez:pi1TV}). In the very recent \cite{Braun}, L.~Braun proves that $\pi_1(X_{\rm{reg}},x)$ is finite, for
the regular locus of a weak Fano variety $X$. If, in addition, $X$ is assumed $\Q$-factorial,
\cite[Cor.~1.3.2]{BCHMcK} and \cite[Thm.~1.1]{GOST} prove that $X$ is a MDS, providing a large class of MDS admitting
finite fundamental group and showing that such hypotheses could be not so restrictive for varieties of interest in
the present paper.
\end{remark}

\subsection{The universal \ét covering}\label{ssez:esistenza}

By analogy with the complex case an algebraic variety
is called \emph{simply connected} if $\pet(X,x)$ is trivial, for some (hence for every) point $x\in X$.

\begin{remark} For $\K=\C$, $\pi_1(X^{\text{an}},x)\cong\{1\}$ implies that $\pet(X,x)\cong\{1\}$, but the converse
does not hold in general, as a non-trivial group can admit a trivial pro-finite completion: a standard example is
given by $\Q$, as $\Q$ does not admit any finite index subgroup. If needed, to avoid confusion in the complex case we
will say either $X$ is \emph{analytically} simply connected or $X^{\text{an}}$ is simply connected if
$\pi_1(X^{\text{an}},x)\cong\{1\}$. But, as observed in Remark~\ref{rem:K=C+finite}, under the further hypothesis that
$\pi_1(X^{\text{an}},x)$ is finite, the converse is also true and one can assert that
\begin{equation*}
  \pi_1(X^{\text{an}},x)\cong\{1\}\ \Longleftrightarrow\ \pet(X,x)\cong\{1\}
\end{equation*}
\end{remark}

\begin{definition}\label{def:univ.cov.}
 Given an algebraic variety $X$, a \emph{universal \ét covering} of $X$ is a simply connected algebraic variety $\widehat{X}$ which is a Galois \ét covering of $X$.
\end{definition}

\begin{notation} Let $\{X_i\stackrel{\phi_i}\to X\}_{i\in \mathfrak{I}}$ be the class of all the Galois coverings of
$X$. Then set
  \begin{equation*}
    \forall\,i,j\in\mathfrak{I}\quad\{X_j\stackrel{\phi_j}{\longrightarrow}X\}\leq
    \{X_i\stackrel{\phi_i}{\longrightarrow}X\}\ :\Longleftrightarrow\ \exists \phi_{ij}:X_i\longrightarrow X_j\ :\
    \phi_i=\phi_j\circ\phi_{ij}
  \end{equation*}
\end{notation}

\begin{proposition}[Universal property]\label{prop:universale}
Assume there exists the universal \ét covering $\widehat{\phi}:\widehat{X}\longrightarrow X$ of an algebraic variety $X$. Then, for every Galois \ét covering $\phi':X'\longrightarrow X$ there exists  a finite \ét morphism $\psi:\widehat{X}\longrightarrow X'$ such that $\widehat{\phi}=\phi'\circ\psi$ and $\psi$ is a Galois \ét covering of $X'$. In particular, the universal \ét covering of $X$ is unique.
\end{proposition}

\begin{proof}
Let $Y$ be a connected component of the fiber product $X'\times_X \widehat{X}$. Then there is an induced Galois \ét covering $\phi:Y\longrightarrow X$ \cite[Prop.~5.3.8-9, Prop.~5.4.6]{Szamuely} making commutative the following diagram
\begin{equation}\label{prod-fibrato}
  \xymatrix{&Y\ar[dd]^{\phi}\ar[dl]_{\text{pr}_1}\ar[dr]^{\text{pr}_2}&\\
    X'\ar[dr]_{\phi'}&&\widehat{X}\ar[dl]^{\widehat{\phi}}\\
    &X&
    }
\end{equation}
By the inverse limit
pro-representation of $\pet(\widehat{X},\widehat{x})$ one gets a natural surjection
\begin{equation*}
  \xymatrix{\{1\}\cong\pet(\widehat{X},\widehat{x})\ar@{>>}[r]&\Aut(\text{pr}_2)}\ \Longrightarrow\ \Aut(\text{pr}_2)\cong\{1\}
\end{equation*}
Then $\text{pr}_2$ is an isomorphism and $\psi:=\text{pr}_1\circ\text{pr}_2^{-1}$ gives a Galois \ét covering of $X'$ such that $\widehat{\phi}=\phi'\circ\psi$.
 \end{proof}

\subsection{The \'{e}tale fundamental group in codimension 1}\label{ssez:pi11}

Let $X$ be an algebraic variety and $x\in X$ a fixed point. Consider the collection of big
Zariski open neighborhoods of $x$ in $X$
\begin{equation*}
  \U1_x:=\{U\subseteq X\ |\ \text{$U$ is open, $x\in U$ and } \codim_X(X\setminus U)>1 \}
\end{equation*}
Consider the partial order relation $\preceq$ on $\U1$ given by setting: $U\preceq V\ :\Leftrightarrow\ U\supseteq
V$\,. Then $(\U1,\preceq)$ is a direct system because any two elements are dominated by their intersection.

\begin{proposition}
  Consider $U,V\in\U1_x$ such that $U\preceq V$. Then there exists a well defined homomorphism
  $\pet(V,x)\longrightarrow\pet(U,x)$.
\end{proposition}

\begin{proof}
  Apply Proposition~\ref{prop:f*} to the open embedding $V\hookrightarrow U$.
\end{proof}

\begin{definition}[The \'{e}tale fundamental group in codimension 1]\label{def:etalep11}
  Let $X$ be an algebraic variety and $x\in X$ a base point. The following inverse limit
  \begin{equation*}
     \pet(X,x)^{(1)}:=\varprojlim_{U\in\U1_x}\pet(U,x)
  \end{equation*}
  is called the \emph{\'{e}tale fundamental group in codimension 1 of $X$ centered at $x$}.
\end{definition}

\begin{remark}
  For $\K=\C$, by the Riemann Existence Theorem of Grothendieck, the \ét fundamental group defined in
  Definition~\ref{def:etalep11} is the pro-finite completion of the \emph{fundamental group in codimension 1}
  $\pi_1^1(X^{\text{an}},x)$ defined in \cite[Def.~3.1]{Buczynska}, that is
  \begin{equation*}
    \pet(X,x)^{(1)}:=\varprojlim_{U\in\U1_x}\widehat{\pi}_1(U^{\text{an}},x)=\widehat{\pi}^1_1(X^{\text{an}},x)
  \end{equation*}
  Therefore if $\pi^1_1(X^{\text{an}},x)$ is finite then $\pet(X,x)^{(1)}\cong\pi^1_1(X^{\text{an}},x)$\,.
\end{remark}
It makes then sense to set the following definition even when $\K$ is an arbitrary algebraically closed field  with
$\Char\K=0$\,:

\begin{definition}[$x$-1-connectedness]
   Let $X$ be an algebraic variety and $x\in X$ be a fixed base point. Then $X$ is called
   \emph{locally simply connected in codimension 1 near to $x$} (or \emph{$x$-1-connected} for ease) if $\pet(X,x)^{(1)}$ is
   trivial.
\end{definition}

\subsection{The direct system of local Galois 1-coverings}\label{ssez:Galois 1-cov}

Consider the collection
$$\{\phi_i:X_i\longrightarrow X\}_{i\in\mathfrak{I}_x^{(1)}}$$
of all Galois 1-coverings of $X$ such that $x\in X\setminus\Br\phi_i$, for every $i\in\mathfrak{I}_x^{(1)}$. Call such
a 1-covering a \emph{local Galois 1-covering of $X$ centered at $x$}.

\begin{proposition}\label{prop:pi11}
  Let $X$ be an algebraic variety and $x\in X$ a base point. Then the set of all local Galois 1-coverings
  of $X$ centered at $x$ is a direct system and
  \begin{equation*}
    \pet(X,x)^{(1)}=\varprojlim_{i\in\mathfrak{I}_x^{(1)}}\Aut^{(1)}(\phi_i)
  \end{equation*}
  where $\Aut^{(1)}(\phi_i)$ is defined in Definition~\ref{def:1-covering}.
\end{proposition}

\begin{proof}
  As for the direct system of \ét coverings, set
  \begin{equation*}
    \forall\,i,j\in\mathfrak{I}_x^{(1)}\quad\{X_j\stackrel{\phi_j}{\longrightarrow}X\}\leq
    \{X_i\stackrel{\phi_i}{\longrightarrow}X\}\ :\Longleftrightarrow\ \exists \phi_{ij}:X_i\longrightarrow X_j\ :\
    \phi_i=\phi_j\circ\phi_{ij}
  \end{equation*}
  defining an order relation on the considered set of local Galois 1-coverings. Moreover, it turns out to be a direct
  system since, recalling diagram (\ref{prod-fibrato}) and considerations there given, there exists a connected component $Y$ of the fibred product $X_j\times_X X_i$ which
  is still a local Galois 1-covering of $X$ centered at $x$.
  The 1-covering morphism $\phi_{ij}:X_i\longrightarrow X_j$ clearly induces a surjection on fibres
  $$\phi_i^{-1}(x)=F^x(\phi_i)\twoheadrightarrow F^x(\phi_j)=\phi_j^{-1}(x)$$
  and then a morphism on the associated
  automorphism groups
  \begin{equation*}
    \xymatrix{\Aut^{(1)}(\phi_i)\ar[r]& \Aut^{(1)}(\phi_j)}
  \end{equation*}
  where $U_i:={X\setminus\Br\phi_i}$ and $U_j:={X\setminus\Br\phi_j}$. Then their inverse limits are well defined and
  the statement follows immediately by Definition~\ref{def:etalep11}.
\end{proof}

\begin{definition}[Universal local 1-covering]\label{def:local 1-covering}
  Let $X$ be an algebraic variety and $x\in X$ be a fixed base point. A local Galois
  1-covering $\widehat{\phi}:\widehat{X}\longrightarrow X$, of $X$ centered at $x$, is called \emph{universal} if it is a $\widehat{x}$-1-connected algebraic variety, for some (hence every) $\widehat{x}\in\widehat{\phi}^{-1}(x)$.
\end{definition}

\begin{proposition}[Universal property]
Let $X$ be an algebraic variety and assume there exists the universal local 1-covering centered at $x\in X$, $\widehat{\phi}:\widehat{X}\longrightarrow X$. Then, for every local Galois 1-covering $\phi':X'\longrightarrow X$ there exists  a finite \ét morphism $\psi:\widehat{X}\longrightarrow X'$ such that $\widehat{\phi}=\phi'\circ\psi$ and $\psi$ is a local Galois 1-covering of $X'$. In particular, the universal local 1-covering of $X$ is unique.
\end{proposition}

\begin{proof}
Reasoning as in the proof of Proposition~\ref{prop:universale}, the choice of a connected component $Y$, of the fiber product $X'\times_X \widehat{X}$, gives a local Galois 1-covering centered at $x$, making commutative a diagram like (\ref{prod-fibrato}). Then, the inverse limit
pro-representation of $\pet(\widehat{X},\widehat{x})^{(1)}$ given by Proposition~\ref{prop:pi11} ends up the proof.
\end{proof}

\begin{proposition}\label{prop:1-connesso&simplyconn}
   Let $X$ be an algebraic variety and assume there exists the universal local 1-covering centered at $x\in X$, $\widehat{\phi}:\widehat{X}\longrightarrow X$. Then $\widehat{X}$ is simply connected.
\end{proposition}

\begin{proof}
  Let $\phi':X'\longrightarrow \widehat{X}$ be any Galois \ét covering of $\widehat{X}$. Then
  $$\xymatrix{\phi:=\phi'\circ\widehat{\phi}:X'\ar[r]& X}$$
  is a Galois 1-covering of $X$ such that $\Br\phi=\Br\widehat{\phi}=:C$.
  $\widehat{\phi}$ is universal, meaning that there exists a
  Galois 1-covering $\psi:\widehat{X}\longrightarrow X'$ such that $\phi\circ\psi=\widehat{\phi}$, that is the following diagram
  commutes
  \begin{equation*}
    \xymatrix{X'\ar[dr]^-\phi\ar[d]^-{\phi'}&\widehat{X}\ar[l]_-\psi\ar[d]^-{\widehat{\phi}}\\
                \widehat{X}\ar[r]^-{\widehat{\phi}}&X}
  \end{equation*}
  Then $\phi'\circ\psi\in\Aut(\widehat{\phi}|_{\widehat{X}_C})$, where $\widehat{X}_C:=\widehat{\phi}^{-1}(X\setminus C)$. Then $\phi'$ restricts to
  give an isomorphism on the big open subset $X'_C:=\phi^{-1}(X\setminus C)\subseteq X'$, meaning that $\phi'$
  gives actually an isomorphism $X'\cong \widehat{X}$, as $\phi'$ is \ét. Then $\Aut(\phi')\cong\{1\}$. Passing to the
  inverse limit on the direct system of Galois \ét coverings of $\widehat{X}$, one gets $\pet(\widehat{X},\widehat{x})\cong\{1\}$, for every
  $\widehat{x}\in \widehat{X}$.
\end{proof}

We are now in a position to give some further analogous results to those given in \cite[\S~3]{Buczynska}.

\begin{corollary}[Compare with Cor.~3.9 in \cite{Buczynska}]\label{cor:pi1smooth}
  If $X$ is a smooth algebraic variety then $\pet(X,x)^{(1)}\cong \pet(X,x)$, for every $x\in X$.
\end{corollary}

\begin{proof}
  By definition
  $$\pet(X,x)^{(1)}=\varprojlim_{X_0\ \stackrel{\text{big}}{\subset}\ X} \pet(X_0,x)\
  \stackrel{\text{Thm.\ref{thm:excisione}}}{\cong}\ \varprojlim_{X_0\ \stackrel{\text{big}}{\subset}\ X} \pet(X,x)\
  =\ \pet(X,x)$$
\end{proof}

\begin{theorem}[Compare with Cor.~3.10 in \cite{Buczynska}]\label{thm:pi1normale}
  Let $X$ be a normal algebraic variety and $X_{\text{\rm reg}}\subseteq X$ be the Zariski open subset of
  regular points of $X$. Then
  $$\forall\,x\in X_{\text{\rm reg}}\quad \pet(X_{\text{\rm reg}},x)\cong\pet(X,x)^{(1)}$$
\end{theorem}

\begin{proof}
First of all let us define a homomorphism
 \begin{equation*}
   \iota: \pet(X_{\text{\rm reg}},x)\longrightarrow \pet(X,x)^{(1)}
 \end{equation*}
 induced by the inclusion $i:X_{\text{\rm reg}}\hookrightarrow X$. By Lemma~\ref{lem:estensione}, a Galois covering $$\phi:U\longrightarrow X_{\rm{reg}}$$
 can be extended to a 1-covering $\overline{\phi}:\overline{U}\longrightarrow X$, where $\overline{U}$ is a normal algebraic variety. By definition of inverse limit, there exists a homomorphism  $\pi:\pet(X_{\text{\rm reg}},x)\longrightarrow \Aut(\phi)$. Given any class $\vf\in \pet(X_{\text{\rm reg}},x)$ consider its representative $f=\pi(\vf)\in\Aut(\phi)$. Choose a point in the fiber $x'\in F^x(\phi):=\phi^{-1}(x)$ and set $x''=f(x')\in F^x(\phi)$. Since $\overline{\phi}$ is an extension of $\phi$, $F^x(\phi)=F^x(\overline{\phi})$ and we can consider an automorphism $\overline{f}\in\Aut^{(1)}(\overline{\phi})$ such that $\overline{f}(x')=x''\in F^x(\overline{\phi})$. Recalling the inverse limit description of $\pet(X,x)^{(1)}$ given by Proposition~\ref{prop:pi11}, define
 \begin{equation}\label{iota}
   \iota[f]:=[\overline{f}]_1\in\pet(X,x)^{(1)}
 \end{equation}
 where $[\overline{f}]_1$ denotes the class determined by $\overline{f}$ in $\varprojlim_{i\in\mathfrak{I}_x^{(1)}}\Aut^{(1)}(\phi_i)$.

 \emph{$\iota[f]$ is well defined}. There are some choices in the definition of $\iota[f]$.
 \begin{itemize}
   \item If one chooses a different automorphism $\overline{f}'\in \Aut^{(1)}(\overline{\phi})$ such that $\overline{f}'(x')=x''\in F^x(\overline{\phi})$ then $\overline{f}'(x')=\overline{f}(x')$ and  Proposition~\ref{prop:lifting} ensures that $\overline{f}|_U=f=\overline{f}'|_U$, as $U$ is connected. That is enough to guarantee that $\overline{f}=\overline{f}'$.
   \item Choose a different Galois covering $\phi':U'\longrightarrow X_{\rm{reg}}$ extending to a 1-covering $\overline{\phi}':\overline{U}'\longrightarrow X$ and let $\pi':\pet(X_{\text{\rm reg}},x)\longrightarrow \Aut(\phi')$ be the associated homomorphism in the inverse limit construction. Set $f':=\pi'(\vf)$ and let $\overline{f}'\in \Aut^{(1)}(\overline{\phi}')$ be an induced automorphism. Then
       $$[\overline{f}]_1=[\overline{f}']_1\in \pet(X,x)^{(1)}$$
       as the 1-coverings $\overline{\phi}$ and $\overline{\phi}'$ are both dominated by a 1-covering $\widetilde{\phi}:Y\longrightarrow X$, being $Y$ a suitable connected component in the fibred product $\overline{U}\times_X \overline{U}'$.
 \end{itemize}

 \emph{$\iota$ is an homomorphism}. Choose $\vf,\psi\in\pet(X_{\text{\rm reg}},x)$ and set
 $$f=\pi(\vf)\,,\ g=\pi(\psi)\in \Aut(\phi)$$
 Then $\pi(\vf\cdot\psi)=f\circ g\in \Aut(\phi)$ so that $\overline{f\circ g}=\overline{f}\circ\overline{g}\in \Aut^{(1)}(\overline{\phi})$ and
 \begin{equation*}
   \iota(\vf\cdot\psi)=[\overline{f}\circ\overline{g}]_1=[\overline{f}]_1\cdot[\overline{g}]_1\in \pet(X,x)^{(1)}
 \end{equation*}

 \emph{$\iota$ is injective}. In fact, if $\vf,\vf'\in  \pet(X_{\text{\rm reg}},x)$ are such that $\vf\neq\vf'$ then there exists a Galois covering $\phi_i:U_i\longrightarrow X_{\rm{reg}}$ such that $f:=\pi_i(\vf)\neq\pi_i(\vf')=:f'$ are different automorphism of $\phi_i$, being $\pi_i$ the canonical projection
 \begin{equation*}
   \pi_i:\pet(X_{\text{\rm reg}},x)\longrightarrow\Aut(\phi_i)
 \end{equation*}
 Let $\phi_j:U_j\longrightarrow X_{\rm{reg}}$ be a Galois covering. Then there exists a connected component $Y\subseteq U_i\times_{X_{\rm{reg}}} U_j$ and an induced commutative diagram
 \begin{equation*}
   \xymatrix{Y\ar[rd]^-\psi\ar[r]^-{\widehat{\phi}_i}\ar[d]_-{\widehat{\phi}_j}&U_j\ar[d]^-{\phi_j}\\
                U_i\ar[r]_-{\phi_i}&X_{\rm{reg}}}
 \end{equation*}
 where $\widehat{\phi}_i, \widehat{\phi}_j, \psi$ are Galois \ét covering. Then, the inverse limit construction gives a canonical projection $\pi_Y: \pet(X_{\text{\rm reg}},x)\longrightarrow\Aut(\psi)$ and a morphism $\widehat{\phi}_{j*}:\Aut(\psi)\longrightarrow\Aut(\phi_i)$ giving the commutativa diagram
 \begin{equation*}
   \xymatrix{\pet(X_{\text{\rm reg}},x)\ar[r]^-{\pi_Y}\ar[rd]_-{\pi_i}&\Aut(\psi)\ar[d]^-{\widehat{\phi}_{j*}}\\
                &\Aut(\phi_i)}
 \end{equation*}
 so that
 \begin{equation}\label{disuguaglianza}
   \widehat{\phi}_{j*}\circ\pi_Y(\vf)=\pi_i(\vf)=f\neq f'=\pi_i(\vf')= \widehat{\phi}_{j*}\circ\pi_Y(\vf')\ \Longrightarrow\ \pi_Y(\vf)\neq\pi_Y(\vf')
 \end{equation}
Again Lemma~\ref{lem:estensione} guarantees that $\widehat{\phi}_i:Y\longrightarrow U_j$, being a Galois \ét covering of $U_j$, extends to a Galois \ét covering $\overline{\widehat{\phi}}_i:\overline{Y}\longrightarrow\overline{U}_j$ so that
$$\overline{\psi}:=\overline{\phi}_j\circ \overline{\widehat{\phi}}_i:\overline{Y}\longrightarrow X_{\rm{reg}}$$
is a Galois 1-covering of $X_{\rm{reg}}$. Then, by the definition (\ref{iota}) of $\iota$,
\begin{equation*}
  \iota(\vf)=[\overline{\pi_Y(\vf)}]_1\in\pet(X,x)^{(1)}\ ,\quad \iota(\vf')=[\overline{\pi_Y(\vf')}]_1\in\pet(X,x)^{(1)}
\end{equation*}
If $\iota(\vf)=\iota(\vf')$, by the inverse limit representation given by Proposition~\ref{prop:pi11}, up to replace $Y$ by a suitable connected étale covering of itself, then $\pi_Y(\vf)$ and $\pi_Y(\vf')$ would be obtained as restrictions of a same automorphism. Then, Proposition~\ref{prop:lifting} and connectedness of $Y$ would give $\pi_Y(\vf)=\pi_Y(\vf')$, against the conclusion in (\ref{disuguaglianza}).

\emph{$\iota$ is surjective}. Consider a class $\vf\in \pet(X,x)^{(1)}$ and let $\overline{\phi}:\overline{U}\longrightarrow X$ be a Galois 1-covering of $X$ extending a Galois \ét covering $\phi:U\longrightarrow X_{\rm{reg}}$, as in Lemma~\ref{lem:estensione}. Let $\pi_{\overline{U}}:\pet(X,x)^{(1)}\longrightarrow \Aut^{(1)}(\overline{\phi})$ be the canonical projection. Then $\pi_{\overline{U}}(\vf)\in \Aut^{(1)}(\overline{\phi})$ induces an automorphism $f\in \Aut(\phi)$ such that
$$\vf=[\pi_{\overline{U}}(\vf)]_1=\iota[f]$$
by the definition (\ref{iota}),
\end{proof}

\begin{remark}\label{rem:K=C+normal}
  For $\K=\C$, what observed in Remark~\ref{rem:K=C+finite}, \wrt the excision property given by
  Theorem~\ref{thm:excisione}, applies also to Corollary~\ref{cor:pi1smooth} and
  Theorem~\ref{thm:pi1normale}: in general they do not imply the analogous Buczy{\'n}ska's results, unless when
  $\pi_1(X,x)^{(1)}$, $\pi_1(X,x)$ and $\pi_1(X_{\rm{reg}},x)$ are assumed to be finite groups.
\end{remark}

Theorem~\ref{thm:pi1normale} allows us to drop local conditions for 1-coverings of a normal variety
$X$, when base points are chosen in the big open $X_{\rm{reg}}$ of regular points. Namely we get the following
consequences.

\begin{corollary}\label{cor:nobasepoint}
  Let $X$ be a normal algebraic variety with $\pet(X_{\text{\rm reg}},x)$ being a finite group for any regular point $x\in X_{\rm{reg}}$. Then
  $$\pet(X,x)^{(1)}\cong\pet(X,x')^{(1)}$$
  for every $x,x'\in X_{\rm{reg}}$.
\end{corollary}

\begin{proof}
  By Theorem~\ref{thm:pi1normale} and Proposition~\ref{prop:puntobase}, one has
  \begin{equation*}
    \pet(X,x)^{(1)}\cong\pet(X_{\rm{reg}},x)\cong\pet(X_{\rm{reg}},x')\cong\pet(X,x')^{(1)}
  \end{equation*}
\end{proof}

\begin{corollary}\label{cor:univ-1-cov}
  Let $\phi':X'\longrightarrow X$ be a Galois 1-covering of a normal algebraic variety $X$. Then $\phi'$
  is the universal 1-covering of $X$ if and only the open subset \emph{$X'_{\rm{reg}}\subseteq X'$} of regular
  points is simply connected, that is $\pet(X'_{\rm{reg}},x')\cong\{1\}$ for some (hence every) $x'\in
  X'_{\rm{reg}}$.

  In other words, $\phi':X'\longrightarrow X$ is the universal 1-covering if and only if it is the universal local
  Galois 1-covering of $X$ centered at any regular point of $X$.
\end{corollary}
\begin{proof} The statement immediately follows by Definition~\ref{def:local 1-covering} and Theorem~\ref{thm:pi1normale}.
\end{proof}

\begin{remark}
  For $\K=\C$, the analogous property of Corollary~\ref{cor:nobasepoint}, on the fundamental groups of $X^{\text{an}}$
  with different base points, is not directly implied by the algebraic statement on their pro-finite completions.
  Anyway, it is a straightforward consequence of path connectedness of $X^{\text{an}}$.

  \noindent On the contrary, Corollary~\ref{cor:univ-1-cov} implies the analogous statement on topological
  1-coverings of $X^{\text{an}}$ under the further hypothesis that $\pi_1(X'^{\text{an}}_{\rm{reg}}, x')$ is finite,
  since $\pet(X'_{\rm{reg}},x')\cong\{1\}$ if and only if $\pi_1(X'^{\text{an}}_{\rm{reg}},x')\cong\{1\}$. Then
  Corollary~\ref{cor:univ-1-cov} gives a proof of what stated in \cite[Rem.~3.14]{Buczynska}, under the further
  hypothesis that $X$ is normal.

\end{remark}

\subsection{Pull back of divisors}\label{ssez:pullback}
Let $X$ be a normal algebraic variety of dimension $n$. The group of
Weil divisors on $X$ is denoted by $\Div(X)$\,: it is the free group generated by prime divisors of $X$. For
$D_1,D_2\in\Div(X)$, $D_1\sim D_2$ means that they are linearly equivalent. The subgroup of Weil divisors linearly
equivalent to 0 is denoted by $\Div_0(X)\leq\Div(X)$. The quotient group $\Cl(X):=\Div(X)/\Div_0(X)$ is called the
\emph{class group}, giving the following short exact sequence of $\Z$-modules
\begin{equation}\label{Wdivisori}
  \xymatrix{0\ar[r]&\Div_0(X)\ar[r]&\Div(X)\ar[r]^-{d_X}&\Cl(X)\ar[r]&0}
\end{equation}
Given a divisor $D\in\Div(X)$, its class $d_X(D)$ is often denoted by $[D]$, when no confusion may arise.

Consider a dominant morphism $\phi:Y\to X$ of normal algebraic varieties whose image $\phi(Y)$ contains a big open subset of $X$ and assume that, for any small closed subset $C\subset X$, the pre-image $\phi^{-1}(C)\subset Y$ is a small closed subset, as well. Then a pull back $\phi^\#$ is
well defined on Cartier divisors by pulling back local equations. This procedure sends principal divisors to
principal divisors, so defining a pull back homomorphism $\phi^*:\Pic(X)\to\Pic(Y)$, where $\Pic$ denotes the group of
linear equivalence classes of Cartier divisors. The given hypotheses on $\phi,Y$ and $X$ allow us to extend the
definition of $\phi^\#$ to every Weil divisor as follows:
\begin{equation}\label{pullback}
  \forall\,D\in\Div(X)\quad \phi^\#(D):=\overline{\phi^\#(D\cap X_{\rm{reg}})}\in\Div(Y)
\end{equation}
(see e.g. \cite[Remark~1.3.4.1]{ADHL}).
Notice that $D\cap X_{\rm{reg}}$ is a Cartier divisor on $X_{\rm{reg}}$; then $\phi^\#(D\cap X_{\rm{reg}})$ is a
Cartier divisor in $Y_{\rm{reg}}\cap \phi^{-1}(X_{\rm{reg}})$ which is a Zariski open subset of $Y$. Clearly
$\phi^\#:\Div(X)\to\Div(Y)$, as defined in (\ref{pullback}), sends Cartier divisors to Cartier divisors and principal
divisors to principal divisors, so giving a well defined pull back homomorphism $\phi^*:\Cl(X)\to\Cl(Y)$ such that
$\phi^*|_{\Pic(X)}$ is the pull back of Cartier divisors defined above.

In the case $\phi:Y\longrightarrow X$ is a 1-covering of normal algebraic va\-rie\-ties obtained as a geometric quotient of a finite abelian group $G$ acting on $Y$, that is $X\cong Y/G$, then hypotheses given above are satisfied and the pre-image
$\phi^{-1}(D)\subseteq Y$ of a Weil divisor $D\in\Div(X)$ is still a Weil divisor of $Y$, meaning that the pull back
defined by (\ref{pullback}) can be easily rewritten by setting
\begin{equation}\label{Weilpullback}
  \phi^\#(D)=\phi^{-1}(D)
  \end{equation}

\section{Application to toric varieties}\label{sez:toric}

The present section is meant to applying results stated in section~\ref{sez:1-covering} to the case of toric
varieties, so generalizing to every algebraically closed field $\K$, with $\Char\K=0$, results given in
\cite[\S~4]{Buczynska} and in \cite{RT-QUOT} under the assumption $\K=\C$.

\subsection{Preliminaries and notation on toric varieties}

Throughout the present paper we will adopt the following definition of a toric variety:
\begin{definition}[Toric variety]\label{def:TV}
  A \emph{toric variety} is a tern $(X,\T,x_0)$ such that:
\begin{itemize}
  \item[(i)] $X$ is a normal, $n$-dimensional, algebraic variety over an algebraically closed field
      $\K$ with $\Char \K=0$,
  \item[(ii)] $\T\cong(\K^*)^n$ is a $n$-torus freely acting on $X$,
  \item[(iii)] $x_0\in X$ is a special point called the \emph{base point}, such that the orbit map $t\in\T\mapsto
      t\cdot x_0\in\T\cdot x_0\subseteq X$ is an open embedding.
\end{itemize}
\end{definition}
For standard notation on toric varieties and their defining \emph{fans} we refer to the extensive treatment
\cite{CLS}.

\begin{definition}[Morphism of toric varieties]\label{def:TVmorphism} Let $Y$ and $X$ be toric varieties with acting
tori $\T_Y$ and $\T_X$ and base points $y_0$ and $x_0$, respectively. A morphism of algebraic varieties
$\phi:Y\longrightarrow X$ is called a \emph{morphism of toric varieties} if
  \begin{itemize}
    \item[(i)] $\phi(y_0)=x_0$\,,
    \item[(ii)] $\phi$ restricts to give a homomorphism of tori $\phi_\T:\T_Y\longrightarrow\T_X$ by setting
        $$\phi_\T(t)\cdot x_0=\phi(t\cdot y_0)$$
  \end{itemize}
\end{definition}

Conditions (i) and (ii) are equivalent to require that $\phi$ induces a morphism between underling fans,
as defined e.g in \cite[\S~3.3]{CLS}.

\subsubsection{List of notation}\label{sssez:lista}
\begin{eqnarray*}
  &M,N,M_{\R},N_{\R}& \text{denote the \emph{group of characters} of $\T$, its dual group}\\
  && \text{and their tensor products with $\R$, respectively;} \\
  &\Si\subseteq \mathfrak{P}(N_{\R})& \text{is the fan defining $X$;}\\
  &&\text{$\mathfrak{P}(N_{\R})$ denotes the power set of $N_{\R}$} \\
  &\Si(i)& \text{is the \emph{$i$--skeleton of $\Si$};}\\
  &\langle\v_1,\ldots,\v_s\rangle\subseteq\N_{\R}& \text{cone generated by $\v_1,\ldots,\v_s\in N_{\R}$;}\\
  && \text{if $s=1$ this cone is called the \emph{ray} generated by $\v_1$;} \\
  &\mathcal{L}(\v_1,\ldots,\v_s)\subseteq N& \text{sublattice spanned by $\v_1,\ldots,\v_s\in N$\,;}\\
\end{eqnarray*}
Let $A\in\mathbf{M}(d,m;\Z)$ be a $d\times m$ integer matrix, then
\begin{eqnarray*}
  &\mathcal{L}_r(A)\subseteq\Z^m& \text{is the sublattice spanned by the rows of $A$;} \\
  &\mathcal{L}_c(A)\subseteq\Z^d& \text{is the sublattice spanned by the columns of $A$;} \\
  &A_I\,,\,A^I& \text{$\forall\,I\subseteq\{1,\ldots,m\}$ the former is the submatrix of $A$ given by}\\
  && \text{the columns indexed by $I$ and the latter is the submatrix}\\
  && \text{of $A$ whose columns are indexed by the complementary }\\
  && \text{subset $\{1,\ldots,m\}\setminus I$;} \\
  &\text{\emph{positive}}& \text{a matrix (vector) whose entries are non-negative.}
\end{eqnarray*}
Given a matrix $V=(\v_1\cdots\v_{m})\in\mathbf{M}(n,m;\Z)$ , then
\begin{eqnarray*}
  &\langle V\rangle=\langle\v_1,\ldots,\v_{m}\rangle\subseteq N_{\R}& \text{is the cone generated by the columns of
  $V$;} \\
  &\SF(V)=\SF(\v_1,\ldots,\v_{m})& \text{is the set of all rational simplicial fans $\Si$ such that}\\
  && \text{$\Sigma(1)=\{\langle\v_1\rangle,\ldots,\langle\v_{m}\rangle\}\subseteq N_{\R}$ and} \\ &&
  \text{$|\Si|=\langle V\rangle$ \cite[Def.~1.3]{RT-LA&GD}.}\\
  & \I_\Si& :=\{I\subseteq\{1,\dots,m\}\,|\,\langle V_I\rangle\in\Si\}\\
  &\G(V)& \text{is a \emph{Gale dual} matrix of $V$ \cite[\S~3.1]{RT-LA&GD}} \\
\end{eqnarray*}
Given a fan $\Si$ in $N_\R\cong\R^n$, the integer matrix $V=(\v_1\cdots\v_m)\in \mathbf{M}(n,m;\Z)$, whose columns are
primitive generators of the 1-skeleton $\Si(1)=\{\langle\v_1\rangle,\ldots,\langle\v_m\rangle\}$, is called a
\emph{fan matrix} of the toric variety $X(\Si)$. The Gale dual $Q=\G(V)$ of a fan matrix is called a \emph{weight
matrix} of $X(\Si)$.

\subsection{$F,CF,W$-matrices and poly weighted spaces (PWS)}

\begin{definition}[$F,CF$-matrices, Def.~3.10 in \cite{RT-LA&GD}]\label{def:Fmatrice} An \emph{$F$--matrix} is a
$n\times m$ matrix  $V$ with integer entries, satisfying the conditions:
\begin{itemize}
\item[(a)] $\rk(V)=n$;
\item[(b)] $V$ is \emph{$F$--complete} i.e. $\langle V\rangle=N_{\R}\cong\R^n$ \cite[Def.~3.4]{RT-LA&GD};
\item[(c)] all the columns of $V$ are non zero;
\item[(d)] if ${\bf  v}$ is a column of $V$, then $V$ does not contain another column of the form $\lambda  {\bf
    v}$ where $\lambda>0$ is real number.
\end{itemize}
A \emph{$CF$--matrix} is a $F$-matrix satisfying the further requirement
\begin{itemize}
\item[(e)] the sublattice ${\mathcal L}_c(V)\subseteq\Z^n$ is cotorsion free, that is, ${\mathcal L}_c(V)=\Z^n$ or,
    equivalently, ${\mathcal L}_r(V)\subseteq\Z^{m}$ is cotorsion free.
\end{itemize}
A $F$--matrix $V$ is called \emph{reduced} if every column of $V$ is composed by coprime entries
\cite[Def.~3.13]{RT-LA&GD}.
\end{definition}
The most significant example of a reduced $F$-matrix is given by the fan matrix $V$ of a rational and complete fan
$\Sigma$.

\begin{definition}[$W$-matrix, Def.~3.9 in \cite{RT-LA&GD}]\label{def:Wmatrice} A \emph{$W$--matrix} is an $r\times m$
matrix $Q$  with integer entries, satisfying the following conditions:
\begin{itemize}
\item[(a)] $\rk(Q)=r$;
\item[(b)] ${\mathcal L}_r(Q)$ does not have cotorsion in $\Z^{m}$;
\item[(c)] $Q$ is \emph{$W$--positive}, that is, $\mathcal{L}_r(Q)$ admits a basis consisting of positive vectors
    \cite[Def.~3.4]{RT-LA&GD}.
\item[(d)] Every column of $Q$ is non-zero.
\item[(e)] ${\mathcal L}_r(Q)$   does not contain vectors of the form $(0,\ldots,0,1,0,\ldots,0)$.
\item[(f)]  ${\mathcal L}_r(Q)$ does not contain vectors of the form $(0,a,0,\ldots,0,b,0,\ldots,0)$, with $ab<0$.
\end{itemize}
A $W$--matrix is called \emph{reduced} if $V=\G(Q)$ is a reduced $F$--matrix \cite[Def.~3.14, Thm.~3.15]{RT-LA&GD}
\end{definition}
The most significant example of a reduced $W$-matrix $Q$ is given by the weight matrix of a rational and complete fan
$\Si$.

\begin{definition}[Poly weighted space, Def.~2.7 in \cite{RT-LA&GD}]\label{def:PWS} A \emph{poly weighted space} (PWS)
is a $n$--dimensional $\Q$--factorial complete toric variety $X(\Si)$, whose reduced fan matrix $V$ is a $CF$--matrix
i.e. if
\begin{itemize}
  \item $V=(\v_1,\ldots,\v_{m})$ is a $n\times m$ $CF$--matrix,
  \item $\Si\in\SF(V)$.
\end{itemize}
\end{definition}

\subsection{1-coverings of toric varieties}

A priori, a 1-covering $\phi:Y\longrightarrow X$ of a toric variety $X$ need not be an equivariant morphism of toric
varieties and $Y$ may not even be a toric variety. A posteriori, we will see that, actually, this is not the case when
$X$ is a \emph{non-degenerate} toric variety, that is:

\begin{definition}[Non-degenerate toric variety]
  A toric variety $X(\Si)$ is called \emph{non-degenerate} if the support $|\Si|$ spans $N_\R$.
\end{definition}

\begin{remark}\label{rem:fattori torici}
  The following facts are equivalent (see e.g. \cite[Prop.~3.3.9]{CLS}):
  \begin{enumerate}
    \item the support $|\Si|$ spans $N_\R$,
    \item the 1-skeleton $\Si(1)$ spans $N_\R$,
    \item $H^0(X,\cO_X^*)\cong\K^*$,
    \item $X(\Si)$ has no torus factors.
  \end{enumerate}
\end{remark}

\begin{definition}[toric 1-covering]\label{def:toric-1-covering}
  A 1-covering $\phi:Y\longrightarrow X$ between toric varieties $Y$ and $X$ is called a \emph{toric 1-covering} if
  $\phi$ is a morphism of toric varieties in the sense of Definition \ref{def:TVmorphism}.
\end{definition}

\begin{proposition}[see e.g. Thm.~3.2.6 in \cite{CLS}]\label{prop:Tbordo}
  Let $X(\Si)$ be a toric variety and consider the torus embedding $\T\hookrightarrow\T \cdot x_0\subseteq X$. Let
  $x_{\rho}$ be the \emph{distinguished point} of a ray $\rho\in\Si(1)$ (see e.g. \cite[\S~3.2]{CLS}). Let $D_{\rho}$
  be the associated torus invariant divisor i.e. $D_{\rho}=\overline{\T\cdot x_{\rho}}\subseteq X$.
      Then $\bigcup_{\rho\in\Si(1)}D_\rho= X\setminus \T\cdot x_0$\,.
\end{proposition}

\begin{theorem}\label{thm:1-cov-torico}
  Let $X(\Si)$ be a non-degenerate toric variety, $Y$ be a normal algebraic variety and
  $\phi:Y\longrightarrow X$ be a Galois 1-covering. Then $Y$ is a non-degenerate toric variety and $\phi$ is a toric
  1-covering with branching locus $$C=\Br(\phi)\subseteq\bigcup_{\rho\in\Si(1)}D_\rho$$
\end{theorem}

 A proof of this result is deferred to \S~\ref{sssez:proof}, after the proof of the following
 Theorem~\ref{thm:QUOT+}.

\subsection{The \ét fundamental group of a toric variety}\label{ssez:pi1TV}

Let us start by recalling the following Grothendieck's remark.

\begin{theorem}[Cor.~1.2 in Exp.~XI, \cite{SGA1}]
  A normal, rational and complete algebraic variety is simply connected.
\end{theorem}

\begin{corollary}\label{cor:Grothendieck}
  A complete toric variety is simply connected.
\end{corollary}

More general results on the computation of the \ét fundamental group of a toric variety were obtained by Danilov.

\begin{theorem}[Prop.~9.3 in \cite{Danilov}]\label{thm:Danilov}
 Let $X(\Si)$ be a non-degenerate toric variety. Then, for every $x\in X$,
 $$\pet(X,x)\cong N\left/N_\Si\right.$$
 where $N_\Si\subseteq N$ is the sublattice spanned by elements in $|\Si|\cap N$.
\end{theorem}

\begin{remark}
  Recall that a toric variety $X(\Si)$ is complete if and only if $|\Si|=N_\R$. Then Danilov's
  Theorem~\ref{thm:Danilov} implies Corollary~\ref{cor:Grothendieck}, as a particular case.

  \noindent Moreover, recalling Remark~\ref{rem:fattori torici}, up to torus factors, Danilov's Theorem~\ref{thm:Danilov} applies to every toric variety.

  \noindent Finally, notice that, up to torus factors, a toric variety turns out to admit finite (\ét) fundamental group, since
  $N_{\Si}$ is a full sublattice of $N$: for $\K=\C$, the analytic counterpart of Theorem~\ref{thm:Danilov} is proved
  in \cite[Thm.~12.1.10]{CLS}. Then, for $\K=\C$, results of section \S~\ref{sez:1-covering} apply as
  well to the fundamental group of the associated analytic variety $X^{\text{an}}$.
\end{remark}

\subsection{The \ét fundamental group in codimension 1 of a toric variety}

We are now in a position to apply results of \S~\ref{sez:1-covering} and compute the \ét fundamental group of a
toric variety without torus factors.

\begin{theorem}\label{thm:pi11toric}
  Let $X(\Si)$ be a non-degenerate toric variety and let $X_1=X(\Si(1))$ the toric variety whose fan is given by the
  $1$-skeleton $\Si(1)$ of $\Si$. Then $X_1$ is a big open subset of the regular locus $X_{\text{\rm reg}}$ of $X$
  and, for every point $x\in X_1$,
  \begin{equation*}
    \pet(X,x)^{(1)}\cong\pet(X_{\text{\rm reg}},x)\cong\pet(X_1,x)\cong N\left/N_1\right.
  \end{equation*}
  where $N_1\subseteq N$ is the sublattice spanned by $\Si(1)\cap N$.
\end{theorem}

\begin{proof}
  Since $X$ is a normal algebraic variety,  Theorem~\ref{thm:pi1normale} gives the following isomorphism
  \begin{equation}\label{1}
    \pet(X,x)^{(1)}\cong\pet(X_{\rm{reg}},x)
  \end{equation}
  for every regular point $x\in X_{\rm{reg}}$. Notice that $X_1$ is smooth: its fan $\Si(1)$ is regular as
  consisting of 1-dimensional cones, only. Moreover, $X_1$ turns out to be a big open subset of $X$. Then
  $X_1\subseteq X_{\rm{reg}}$ is a big open subset of $X_{\rm{reg}}$, too. By the excision property given by
  Theorem~\ref{thm:excisione}, one has
  \begin{equation}\label{2}
    \pet(X_{\rm{reg}},x)\cong\pet(X_1,x)
  \end{equation}
  for every $x\in X_1$. Finally, since $\Si(1)$ spans $N_\R$, one applies Danilov's Theorem~\ref{thm:Danilov} to get
  \begin{equation}\label{3}
    \pet(X_1,x)\cong N\left/N_1\right.
  \end{equation}
  The proof ends up by putting together (\ref{1}), (\ref{2}) and (\ref{3}).
\end{proof}

\begin{remark}
  For $\K=\C$, the analytic counterpart of Theorem~\ref{thm:Danilov} given by \cite[Thm.~12.1.10]{CLS} shows that
  $\pi_1(X_1^{\text{an}},x)\cong N/N_1$. This suffices to show that the argument proving Theorem~\ref{thm:pi11toric}
  applies to the analytic setup, as well. Then one gets analogous statements for the fundamental group in codimension
  1 of the associated analytic variety $X^{\text{an}}$ and this is what Buczy{\'n}ska did in \cite[\S~4]{Buczynska} for
  any complex toric variety, by obviously adding the contribution of any torus factor.
\end{remark}

\subsection{The universal 1-covering of a non-degenerate toric variety}

 It is a well known fact, already observed in the beginning of \S~\ref{ssez:esistenza}, that in general the universal
 \ét covering of an algebraic variety does not exist. The same clearly holds for the universal (local) 1-covering.
 Therefore exhibiting a class of algebraic varieties admitting either a universal \ét covering or a universal (local,
 in case) 1-covering, is always of some interest. Recently, jointly with Lea Terracini, we proved that $\Q$-factorial
 and complete toric varieties, over the complex field $\C$, always admit a universal 1-covering
 \cite[Thm.~2.2]{RT-QUOT}, which turns out to be still a $\Q$-factorial and complete toric variety, coherently with
 Theorem~\ref{thm:1-cov-torico}. In particular a universal
 1-covering of this kind is always a PWS (in the sense of Definition~\ref{def:PWS}) canonically determined by the
 initially given $\Q$-factorial complete toric variety.

 The present section is meant to generalize this result over the ground field and to extending it to the bigger
 range of \emph{non-degenerate} toric varieties, so dropping both hypothesis of completeness and $\Q$-factoriality.

\begin{theorem}[Compare with Thm.~2.2 and Rem.~2.3 in \cite{RT-QUOT}]\label{thm:QUOT+}
A non-degenerate toric va\-rie\-ty $X$ over an algebraically closed field $\K$ with $\Char\K=0$, admits a universal
1-covering $\vf:\widetilde{X}\longrightarrow X$ which is a toric 1-covering of non-degenerate toric varieties. The
induced pull-back on divisors gives a group epimorphism $\vf^*:\Cl(X)\twoheadrightarrow\Cl(\widetilde{X})$ whose
kernel is $$\ker(\vf^*)\cong\Tors(\Cl(X))\cong\pet(X,x)^{(1)}\cong\pet(X_{\rm{reg}},x)$$
for every regular point $x\in X_{\text{\rm reg}}$.

\noindent In particular every non-degenerate toric variety $X$ can be canonically described as a finite geometric
quotient $X\cong \widetilde{X}/\pet(X,x)^{(1)}$ of the universal 1-covering $\widetilde{X}$ by the torus-equivariant
action of $\pet(X,x)^{(1)}\cong \Tors(\Cl(X))$ on the fibers of $\vf$.

\noindent Moreover, if $V$ is a fan matrix of $X$ then $\widetilde{V}=\G(\G(V))$ is a fan matrix of $\widetilde{X}$.

\noindent By construction $\widetilde{X}$ is $\Q$-factorial (complete) if and only if $X$ is $\Q$-factorial
(complete). In particular, if $X$ is both complete and $\Q$-factorial then its universal 1-covering $\widetilde{X}$ is
a PWS.
\end{theorem}

\begin{corollary}[Rem.~2.4 in \cite{RT-QUOT}, Prop.~3.1.3 in \cite{RT-LA&GD}]\label{cor:QUOT}
   Consider a toric 1-covering $\phi:Y\longrightarrow X$ of a non-degenerate toric variety $X$ over an algebraically
   closed field $\K$ with $\Char\K=0$. If $V$ and $W$ are fan matrices of $X$ and $Y$, respectively, then there exists
   a unique matrix $\b\in\GL_n(\Q)\cap\mathbf{M}(n,n;\Z)$ such that $V=\b\cdot W$.

   Moreover if $X$ is $\Q$-factorial then also $Y$ is, and $\phi^*:\Cl(X){\twoheadrightarrow}\Cl(Y)$ is a group
   epimorphism inducing a $\Q$-module isomorphism
  $$\xymatrix{\Pic(X)\otimes_\Z\Q\cong\Cl(X)\otimes_\Z\Q\ar[r]^-{\phi^*_\Q}_-\cong&
  \Cl(Y)\otimes_\Z\Q\cong\Pic(Y)\otimes_\Z\Q}$$
\end{corollary}

\begin{proof}[Proof of Thm.~\ref{thm:QUOT+}]  Calling $n=\dim X$ and $r=\rk\Cl(X)$, recall the definition of
$\I_\Si\subseteq\mathfrak{P}\{1,\dots,n+r\}$ given in \ref{sssez:lista}. Let $V$ be a fan matrix of $X$. Then
$\Si(1)=\{\langle\v_i\rangle\,|\,\v_i\ \text{is the $i$-th column of}\ V\}$\,. Consider the sublattice $N_1\subseteq
N=\Z^n$ spanned by the $\v_i$'s.
Since $X$ is non-degenerate, the lattice $N_1$ is a full sublattice of $N$ and $N/N_1$ is a finite abelian group. Let
$\widetilde{V}=\G(\G(V))$ be a double Gale dual matrix of $V$ and consider the fan
\begin{equation}\label{Sigmatilde}
  \widetilde{\Si}:=\{\langle \widetilde{V}_I\rangle\,|\,I\in\I_\Si\}\subseteq\mathfrak{P}(N_1)
\end{equation}
defining a toric variety $\widetilde{X}=\widetilde{X}(\widetilde{\Si})$. The natural inclusion $N_1\hookrightarrow
N=\Z^n$ induces a surjection $\widetilde{X}\twoheadrightarrow X$ which turns out to be the canonical projection on the
quotient of $\widetilde{X}$ by the action of the finite abelian group $N/N_1$. Theorems~\ref{thm:pi1normale} and \ref{thm:pi11toric} give that
$$\pet(X,x)^{(1)}\cong\pet(X_{\rm{reg}},x)\cong N/N_1$$
for every $x\in X_{\rm{reg}}$. The following Lemma~\ref{lem:tors} shows that $N/N_1\cong\Tors(\Cl(X))$.
The same argument applied to $\widetilde{X}$ shows that it is 1-connected and $\widetilde{X}\twoheadrightarrow X$
turns out to be the universal 1-covering of $X$. Moreover $\Tors(\Cl(\widetilde{X}))=0$ and
$\rk\Cl(\widetilde{X})=\rk(\Cl(X))=r$. By the construction (\ref{Sigmatilde}) of the fan $\widetilde{\Si}$, one
clearly sees that $\widetilde{X}$ is $\Q$-factorial (complete) if and only if $X$ is.
  \end{proof}

\begin{lemma}[Compare with Thm.~2.4 in \cite{RT-LA&GD}]\label{lem:tors}
Let $X(\Si)$ be a non-degenerate toric variety and $N_1\subseteq N$ be the sublattice spanned by primitive generators
of rays in $\Si(1)$. Then
 \begin{equation*}
   \Tors(\Cl(X))\cong N/N_1
 \end{equation*}
\end{lemma}

\begin{proof} The proof is the same as in \cite[Thm.~2.4]{RT-LA&GD}. Anyway it is here reported to adapting the key
argument to the current weaker hypotheses.

  Let $\Div_\T(X)$ denotes the group of torus invariant Weil divisors. Then there is the following well known short
  exact sequence (see e.g. \cite[Thm.~4.1.3]{CLS})
  \begin{equation*}
    \xymatrix{0\ar[r]&M\ar[r]^-{div}&\Div_\T(X)\ar[r]^-d&\Cl(X)\ar[r]&0}
  \end{equation*}
  Adopting the same notation as in the proof of Thm.~\ref{thm:QUOT+}, this gives
  \begin{equation*}
    \Cl(X)\cong\Div_\T(X)\left/\im(div)\right.\cong\Z^{n+r}\left/\Ls_r(V)\right.
  \end{equation*}
  where $V$ is a fan matrix of $X$ (recall notation introduced in \ref{sssez:lista}). Then
  \begin{equation*}
    \Tors(\Cl(X))\cong\Tors(\Z^{n+r}/\Ls_r(V))\cong\Tors(\Z^{n}/\Ls_c(V))\cong\Z^n/\Ls_r(T_n)
  \end{equation*}
  where $\left(
                   \begin{array}{c}
                     T_n \\
                     \mathbf{0} \\
                   \end{array}
                 \right)$ is the Hermite normal from of the transpose matrix $V^T$. In particular the rows of $T_n$
                 give a basis of $N_1$, meaning that $N/N_1\cong\Z^n/\Ls_r(T_n)$.
\end{proof}

\begin{proof}[Proof of Cor.\ref{cor:QUOT}]
  The first part of the statement follows immediately by \cite[Prop.~3.1.3]{RT-LA&GD} (see also
  \cite[Rem.~2.4]{RT-QUOT}) whose argument is completely $\Z$-linear. The second part is then an immediate consequence
  of Theorem~\ref{thm:QUOT+}.
\end{proof}

\subsubsection{A proof of Theorem~\ref{thm:1-cov-torico}}\label{sssez:proof}

By Theorem~\ref{thm:QUOT+}, $X$ admits a universal 1-covering $\vf:\widetilde{X}\longrightarrow X$ which is a toric
1-covering of non-degenerate toric varieties. Then there exits a Galois 1-covering $f:\widetilde{X}\longrightarrow Y$
such that $\vf=\phi\circ f$\,. In particular this means that there exists a (normal) subgroup $H\leq\Aut(\vf)$ such
that $Y\cong \widetilde{X}/H$ and $\phi$ is the associated quotient projection \cite[Prop.~5.3.8]{Szamuely}. Again
Theorem~\ref{thm:QUOT+} gives that
\begin{equation*}
  \Aut(\vf)\cong\Tors(\Cl(X))\cong\pet(X,x)^{(1)}\cong N/N_1
\end{equation*}
meaning that $H$ corresponds to a sublattice $N_H\leq N$ such that
\begin{equation*}
  N_1\leq N_H\,,\quad H\cong N_H/N_1\cong\Tors(\Cl(Y))\cong\pet(Y,y)^{(1)}
\end{equation*}
for some base point $y\in \phi^{-1}(x)$. Then \cite[Rem.~2.4]{RT-QUOT} shows that there exists an integer matrix
$\eta\in\GL_n(\Q)\cap\mathbf{M}_n(\Z)$ such that $Y$ is the non-degenerate toric variety whose fan matrix is given by
$V^\eta:=\eta\cdot V$ and determined by the following fan
\begin{equation*}
  \Si_\eta:=\{\langle V^\eta_I\rangle\,|\,I\in\I_\Si\}\subseteq\mathfrak{P}(N_H)
\end{equation*}
By construction, $\phi$ is clearly equivariant giving rise to a toric 1-covering.

\section{Application to Mori dream spaces}\label{sez:MDS}

The present section is meant to apply results of sections~\ref{sez:1-covering} and \ref{sez:toric} to
the case of Mori dream spaces. Actually varieties here considered are more general algebraic varieties than Mori dream
spaces as introduced by Hu and Keel in \cite{Hu-Keel}, as we will not require neither any projective embedding nor
completeness when showing main applications. These varieties will be called \emph{weak} Mori dream spaces (wMDS) to
distinguishing them from the usual Hu-Keel Mori dream spaces (MDS) (see Definition~\ref{def:wMDS}).

Next subsections \S~\ref{ssez:Cox} and \S~\ref{ssez:wMDS} will be devoted, the former, to recalling main notation on
Cox rings, essentially following \cite{ADHL}, and the latter, to quickly explain main results about the toric
embedding properties of a wMDS, as studied in \cite{R-wMDS}.

\subsection{Cox sheaf and algebra of an algebraic variety}\label{ssez:Cox}

For what concerns the present topic we will essentially adopt the approach described in the extensive book
\cite{ADHL} and notation introduced in \cite[\S~1.3]{R-wMDS}. The interested reader is referred to those sources for
any further detail.

\subsubsection{Assumption}\label{ipotesi} In the following, $\Cl(X)$ is assumed to be a \emph{finitely generated}
(f.g.) abelian group of rank $r:=\rk(\Cl(X))$. Then $r$ is called either the \emph{rank}
of $X$.
Moreover we will assume that every invertible global function is constant i.e. $H^0(X,\cO_X^*)\cong\K^*$\,.

\subsubsection{Choice}\label{ssez:K} Choose a f.g. subgroup $K\leq\Div(X)$ such that
$$\xymatrix{d_K:=d_X|_K:K\ar@{>>}[r]&\Cl(X)}$$
is an \emph{epimorphism}. Then $K$ is a free group of rank $m\geq r$ and (\ref{Wdivisori}) induces the following exact
sequence of $\Z$-modules
\begin{equation*}
  \xymatrix{0\ar[r]&K_0\ar[r]&K\ar[r]^-{d_K}&\Cl(X)\ar[r]&0}
\end{equation*}
where $K_0:=\Div_0(X)\cap K=\ker(d_K)$.

\begin{definition}[Sheaf of divisorial algebras, Def.~1.3.1.1 in \cite{ADHL}] The \emph{sheaf of divisorial algebras}
associated with the subgroup $K\leq\Div(X)$ is the sheaf of $K$-graded $\cO_X$-algebras
\begin{equation*}
  \cS:=\bigoplus_{D\in K}\cS_D\,,\quad \cS_D:=\cO_X(D)
\end{equation*}
where the multiplication in $\cS$ is defined by multiplying homogeneous sections in the field of functions $\K(X)$.
\end{definition}

\subsubsection{Choice}\label{ssez:chi} Choose a character $\chi:K_0\to\K(X)^*$ such that
\begin{equation*}
  \forall\,D\in K_0\quad D=(\chi(D))
\end{equation*}
where $(f)$ denotes the principal divisor defined by the rational function $f\in\K(X)^*$. Consider the ideal sheaf
$\I_\chi$ locally defined by sections $1-\chi(D)$ i.e.
\begin{equation*}
  \Ga(U,\I_\chi)=\left((1-\chi(D))|_U\,|\,D\in K_0\right)\subseteq\Ga(U,\cS)\,.
\end{equation*}
This induces the following short exact sequence of $\cO_X$-modules
\begin{equation}\label{Sdivisori}
  \xymatrix{0\ar[r]&\I_\chi\ar[r]&\cS\ar[r]^-{\pi_\chi}&\cS/\I_\chi\ar[r]&0}
\end{equation}

\begin{definition}[Cox sheaf and Cox algebra, Construction~1.4.2.1 in \cite{ADHL}]\label{def:CoxRings}
Keeping in mind the exact sequence (\ref{Sdivisori}), the \emph{Cox sheaf} of $X$, associated with $K$ and $\chi$, is
the quotient sheaf $\mathcal{C}ox:=\cS/\I_\chi$ with the $\Cl(X)$-grading
\begin{equation*}
  \cox:=\bigoplus_{\d\in \Cl(X)}\cox_{\d}\,,\quad \cox_{\d}:=\pi_\chi\left(\bigoplus_{D\in
  d_K^{-1}(\d)}\cS_D\right)
\end{equation*}
Passing to global sections, one gets the following \emph{Cox algebra} (usually called Cox \emph{ring}) of $X$,
associated with $K$ and $\chi$,
\begin{equation*}
  \Cox(X):=\cox(X)= \bigoplus_{\d\in \Cl(X)}\Ga(X,\cox_{\d})
\end{equation*}
\end{definition}

\begin{remarks}\label{rems}\hfill
\begin{enumerate}
  \item \cite[Prop.~1.4.2.2]{ADHL} Depending on choices \ref{ssez:K} and \ref{ssez:chi}, both Cox sheaf and algebra
      are not canonically defined. Anyway, given two choices $K,\chi$ and $K',\chi'$ there is a graded isomorphism
      of $\cO_X$-modules $$\cox(K,\chi)\cong\cox(K',\chi')$$
  \item For any open subset $U\subseteq X$, there is a canonical isomorphism
  $$\xymatrix{\Ga(U,\cS)/\Ga(U,\I_\chi)\ar[r]^-{\cong}&\Ga(U,\cox)}$$
  In particular $\Cox(X)\cong H^0(X,\cS)/H^0(X,\I_\chi)$. This fact gives a precise meaning to the usual ambiguous
  writing
  $$\Cox(X)\cong \bigoplus_{[D]\in\Cl(X)}H^0(X,\cO_X(D))$$
\end{enumerate}
\end{remarks}

\subsection{Weak Mori dream spaces (wMDS) and their embedding}\label{ssez:wMDS}

In the lite\-ra\-tu\-re Mori dream spaces (MDS) come with a required projective embedding essentially for their
optimal behavior with respect to the termination of Mori program. As explained in \cite{R-wMDS}, this assumption is
not necessary to obtain main properties of MDS, like e.g. their toric embedding, chamber decomposition of their moving
and pseudo-effective cones and even termination of Mori program, for what this fact could mean for a complete and
non-projective algebraic variety.

According to notation introduced in \cite{R-wMDS}, we set the following

\begin{definition}[wMDS]\label{def:wMDS} A $\Q$-factorial algebraic va\-rie\-ty $X$ sa\-ti\-sfying
assumption~\ref{ipotesi} is called a \emph{weak Mori dream space} (wMDS) if $\Cox(X)$ is a finitely generated
$\K$-algebra. A projective wMDS is called a \emph{Mori dream space} (MDS).
\end{definition}

\subsubsection{Total coordinate and characteristic spaces}

Consider a wMDS $X$ and its Cox sheaf $\cox$. The latter is  \emph{locally of finite type}, that is there
exists a finite affine covering $\bigcup_iU_i=X$ such that $\cox(U_i)$ are finitely generated $\K$-algebras
\cite[Construction~1.3.2.1, Propositions~1.6.1.2, 1.6.1.4]{ADHL}. The \emph{relative spectrum} of $\cox$ \cite[Ex.~II.5.17]{Hartshorne},
      \begin{equation}\label{relspectra}
        \xymatrix{\widehat{X}=\Spec_X(\cox)\ar[r]^-{p_X}& X}
      \end{equation}
        is a normal and quasi-affine variety $\widehat{X}$, coming with an actions of the quasi-torus
        $G:=\Hom(\Cl(X),\K^*)$, whose quotient map is realized  by the canonical morphism $p_X$ in (\ref{relspectra})
        \cite[\S~1.3.2 , Construction~1.6.1.5]{ADHL}. $\widehat{X}$ is called the \emph{characteristic space} of $X$
        and $G$ is called the \emph{characteristic quasi-torus} of $X$.

        Moreover consider
        \begin{equation*}
         \overline{X}:=\Spec(\Cox(X))
        \end{equation*}
which is a normal, affine variety, called the \emph{total coordinate space} of $X$. Then there exists
an open embedding $j_X:\widehat{X}\hookrightarrow\overline{X}$.
       The action of the quasi-torus $G$ extends to $\overline{X}$ in such a way that $j_X$ turns out to be
       equivariant.

\begin{theorem}[Cox Theorem for a wMDS]\label{thm:Cox-wMDS} Let $X$ be a wMDS and consider the natural action of the
quasi-torus $G$ on the total coordinate space $\overline{X}$. Then the loci of stable and semi-stable points coincide
with the open subset $j_X(\widehat{X})\subseteq \overline{X}$, which is the characteristic space of $X$. Then the
canonical morphism $p_X:\widehat{X}\twoheadrightarrow X$ is the associated  $1$-free and geometric quotient. In
particular $$(p_X)_*(\cO_{\widehat{X}})\cong\cox\quad,\quad(p_X)^*:\cO_X\stackrel{\cong}{\longrightarrow}
\cox^G:=(p_X)_*\cO_{\widehat{X}}^G$$
\end{theorem}
For a definition of used notation and a sketch of proof we refer the interested reader to Definitions~8,9,10 and
Theorem~1 in \cite{R-wMDS}.

\subsubsection{Irrelevant loci and ideals}\label{ssez:Irr}
$\Cox(X)$ is a finitely generated $\K$-algebra. Then, up to the choice of a set of generators
$\mathfrak{X}=(x_1,\ldots,x_m)$, we get
       \begin{equation*}
         \Cox(X)\cong\K[\X]/I
       \end{equation*}
       being $I\subseteq\K[\mathfrak{X}]:=\K[x_1,\ldots,x_m]$ a suitable ideal of relations.

       \noindent Calling $\overline{W}:=\Spec\K[\X]\cong\K^m$, the canonical surjection
       \begin{equation}\label{can-surj}
         \xymatrix{\pi_\X:\K[\X]\ar@{->>}[r]&\Cox(X)}
       \end{equation}
       gives rise to a closed embedding $\overline{i}:\overline{X}\hookrightarrow\overline{W}\cong\K^m$, depending on
       the choice of $(K,\chi,\X)$.

\begin{definition}[Irrelevant loci and ideals]\label{def:irr}
  Let $X$ be a wMDS. The \emph{irrelevant locus} of a total coordinate space $\overline{X}$ of $X$ is the Zariski
  closed subset given by the complement $B_X:=\overline{X}\setminus j_X(\widehat{X})$. Since $\overline{X}$ is affine,
  the irrelevant locus $B_X$ defines an \emph{irrelevant ideal} of the Cox algebra $\Cox(X)$, as
  \begin{equation*}
    \mathcal{I}rr(X):=\left(f\in\Cox(X)_\d\,|\, \d\in\Cl(X)\ \text{and}\ f|_{B_X}=0 \right)\subseteq \Cox(X)\,.
  \end{equation*}
  Analogously, after the choice of a set $\X$ of generators of $\Cox(X)$, consider the \emph{lifted irrelevant ideal}
  of $X$
  \begin{equation*}
    \widetilde{\I rr}:=\pi_\X^{-1}(\I rr(X))\subseteq\K[\X]\,.
  \end{equation*}
  The associated zero-locus $\widetilde{B}=\mathcal{V}(\widetilde{\I rr})\subseteq \Spec(\K[\X])=:\overline{W}$ will
  be called the \emph{lifted irrelevant locus} of $X$.
\end{definition}

\subsubsection{The canonical toric embedding}

Let $X$ be a wMDS and $\Cox(X)$ be its Cox ring. Recall that the latter is a graded $\K$-algebra over the class group
$\Cl(X)$ of $X$. Given a set of generators $\X=\{x_1,\ldots,x_m\}$ of $\Cox(X)$ one can always ask, up to
factorization, that their classes $\overline{x}_i$ are \emph{$\Cl(X)$-prime}, in the sense of
\cite[Def.~1.5.3.1]{ADHL}, that is:
\begin{itemize}
  \item a non-zero non-unit
$y\in\Cox(X)$ is $\Cl(X)$-prime if there exists $\d\in\Cl(X)$ such that $y\in\Cox(X)_\d$ (i.e. $y$ is
\emph{homogeneous}) and, for $i=1,2$,
  \begin{equation*}
    \forall\,\d_i\in\Cl(X)\,,\,\forall\,f_i\in\Cox(X)_{\d_i}\quad y\,|\,f_1f_2\ \Longrightarrow\ y\,|\,f_1\
    \text{or}\ y\,|\,f_2
  \end{equation*}
\end{itemize}

\begin{definition}[Cox generators and bases]\label{def:Coxgen}
  Given a wMDS $X$ and a set $\X$ of generators of $\Cox(X)$, an element $x\in\X$  is called a \emph{Cox generator} if
  its class $\overline{x}$ is $\Cl(X)$-prime. If $\X$ is entirely composed by Cox generators then it is called a
  \emph{Cox basis} of $\Cox(X)$ if it has \emph{minimum cardinality}.
\end{definition}

\begin{theorem}[Canonical toric embedding]\label{thm:canonic-emb}
Let $X$ be a wMDS and $\X$ be a Cox basis of $\Cox(X)$. Then there exists a closed embedding $i:X\hookrightarrow W$
into a $\Q$-factorial and non-degenerate toric variety $W$, fitting into the following commutative diagram
\begin{equation}\label{embediag}
    \xymatrix{\overline{X}\ar@/^2pc/[rrr]^-{\overline{i}}&\widehat{X}\ar@{_{(}->}[l]_-{j_X}\ar@{>>}[d]^-{p_X}
    \ar@{^(->}[r]^-{\widehat{i}}
                &\widehat{W}\ar@{^(->}[r]^-{j_W}\ar@{>>}[d]^-{p_W}&\overline{W}\\
                &{X}\ar@{^(->}[r]^-{i}&{W}&}
  \end{equation}
where
\begin{enumerate}
  \item $\overline{W}=\Spec\K[\X]$,
  \item $\widehat{W}:=\overline{W}\backslash\widetilde{B}$ is a Zariski open subset and
      $j_W:\widehat{W}\hookrightarrow\overline{W}$ is the associated open embedding,
  \item $\widehat{i}:=\overline{i}|_{\widehat{X}}$\,,
  \item $p_W:\widehat{W}\twoheadrightarrow W$ is a 1-free geometric quotient by an action of the characteristic
      quasi-torus $G=\Hom(\Cl(X),\K^*)$ on the affine va\-rie\-ty $\overline{W}$, \wrt $\widehat{i}$ turns out to be
      equivariant and $j_W(\widehat{W})$ is the locus of stable and semi-stable points. Moreover
      $(p_W)^*:\cO_W\stackrel{\cong}{\longrightarrow} (p_W)_*\cO_{\widehat{W}}^G$\,.
\end{enumerate}
\end{theorem}

For a proof of this theorem we refer the interested reader to \cite[Thm.~2, Cor.~1]{R-wMDS}. Here we just recall
that, given the Cox basis $\X=\{x_1,\ldots,x_n\}$, the embedding, canonically determined by the surjection
(\ref{can-surj}) between the associated algebras, can be concretely described by evaluating the Cox generators as
follows
 $$\xymatrix{x\in \overline{X}\ar@{|->}[r]&\overline{i}(x):=(x_1(x),\ldots,x_m(x))\in \K^m}$$
Moreover the $G$-action on $\overline{W}$ is defined by observing that the class $\overline{x}_i$ is homogeneous, that
is there exists a class $\d_i\in \Cl(X)$ such that $\overline{x}_i\in\Cox(X)_{\d_i}$. Then one has
$$\xymatrix{(g,z)\in G\times\overline{W}\ar@{|->}[r]&g\cdot z:=(\chi_1(g)z_1,\ldots,\chi_m(g)z_m)\in\overline{W}}$$
where $\chi_i:G\to\K^*$ is the character defined by
$\chi_i(g)=g(\d_i)$\,.

\begin{remarks}
  \begin{enumerate}
  \item The ambient toric variety $W$, defined in Theorem~\ref{thm:canonic-emb}, only depends on the choices of the
      Cox basis $\X$ and no more on $K$ and $\chi$, as given in \ref{ssez:K} and \ref{ssez:chi}. In fact, for
      different choices $K',\chi'$ we get an isomorphic Cox ring, as observed in Remark~\ref{rems}~(1). Then it
      still admits the same presentation $\K[\X]/I$, meaning that the toric embedding $i:X\hookrightarrow W$ remains
      unchanged, up to isomorphism.

      \noindent Actually the toric embedding exhibited in Theorem~\ref{thm:canonic-emb} only depends on the
      cardinality $|\X|$, that is on the choice of a Cox basis instead of a more general set of Cox generators. One can then fix a \emph{canonical toric embedding} $i:X\hookrightarrow W$ as that
      associated, up to isomorphisms, to a Cox basis, that is the one presented in Theorem~\ref{thm:canonic-emb}.
  \item Varieties $\widehat{W}$ and $\overline{W}$, exhibited in Theorem~\ref{thm:canonic-emb}, are called the
      \emph{characteristic space} and the \emph{total coordinate space}, respectively, of the canonical toric
      ambient variety $W$. In particular, the geometric quotient $p_W:\widehat{W}\twoheadrightarrow W$ is precisely
      the classical Cox's quotient presentation of a \emph{non-degenerate} (i.e. not admitting torus factors) toric
      variety \cite{Cox}.
\end{enumerate}
\end{remarks}

\subsubsection{The canonical toric embedding is a neat embedding}
Let $X$ be a wMDS and $i:X\hookrightarrow W$ be its canonical toric embedding constructed in
Theorem~\ref{thm:canonic-emb}. Let $V=(\v_1,\ldots,\v_m)$ be a fan matrix of $W$, which is a representative matrix of
the dual morphism
\begin{equation*}
  \xymatrix{\Hom(\Div_\T(W),\Z)\ar[r]^-{div_W^\vee}_-V&N:=\Hom(M,\Z)}
\end{equation*}
In the following we will then denote
$D_i:=D_{\langle\v_i\rangle}$ the prime torus invariant associated with the ray $\langle\v_i\rangle\in\Si(1)$, for
every $1\leq i\leq m$.

\begin{proposition}[Pulling back divisor classes]\label{prop:pullback} Let $i:X\hookrightarrow W$ be a closed
embedding of a normal algebraic variety X into a toric variety $W(\Si)$ with acting torus $\T$. Let
$D_{\rho}=\overline{\T\cdot x_{\rho}}$, for $\rho\in\Si(1)$, be the invariant prime divisors of $W$ and assume that
$\{i^{-1}(D_\rho)\}_{\rho\in\Si(1)}$ is a set of pairwise distinct hypersurfaces in $X$. Then it is well
defined a pull back homomorphism $i^*:\Cl(W)\to\Cl(X)$.
\end{proposition}
For a proof, the interested reader is referred to \cite[Remark~3.2.5.1]{ADHL} and \cite[Prop.~4]{R-wMDS}.

\begin{definition}[Neat embedding]\label{def:neat}
Let $X$ be a normal algebraic variety and $W(\Si)$ be a toric variety. Let
$\{D_{\rho}\}_{\rho\in\Si(1)}$ be the torus invariant prime divisors of $W$. A closed embedding $i:X\hookrightarrow W$
is called a \emph{neat (toric) embedding} if
  \begin{itemize}
    \item[(i)] $\{i^{-1}(D_\rho)\}_{\rho\in\Si(1)}$ is a set of pairwise distinct  hypersurfaces in $X$,
    \item[(ii)] the pullback homomorphism defined in Proposition~\ref{prop:pullback},
        $$\xymatrix{i^*:\Cl(W)\ar[r]^-{\cong}&\Cl(X)}$$
        is an isomorphism.
  \end{itemize}
\end{definition}

\begin{proposition}\label{prop:neat}
  The canonical toric embedding $i:X\hookrightarrow W$, of a wMDS $X$, is a neat embedding. Moreover the isomorphism
  $i^*:\Cl(W)\stackrel{\cong}{\longrightarrow}\Cl(X)$ restricts to give an isomorphism $\Pic(W)\cong\Pic(X)$.
\end{proposition}
For a proof, the interested reader is referred to \cite[Cor.~3.3.1.7]{ADHL} and \cite[Prop.~5]{R-wMDS}.

\subsubsection{Sharp completions of the canonical ambient toric variety}
E\-ve\-ry algebraic variety can be embedded in a complete one, by Nagata's theorem \cite[Thm.]{Nagata}. For  those
endowed with an algebraic group action Sumihiro provided an equivariant version of this theorem \cite{Sumihiro1},
\cite{Sumihiro2}. In particular, for toric varieties, it corresponds to the Ewald-Ishida  combinatorial completion
procedure for fans \cite[Thm.~III.2.8]{ES}, recently simplified by Rohrer \cite{Rohrer11}. Anyway, all these
procedures in general require the adjunction of some new ray into the fan under completion, that is an increasing of
the rank of $X$. For a toric variety of dimension $\le 3$, whose fan has positive hull filling the whole $N_\R$, a completion which does not increase the number of rays can be found. The latter does no more hold in dimension $\geq 4$: there are examples of 4-dimensional fans of this kind, which cannot be
completed without the introduction of new rays. Consider the Remark ending up \S~III.3 in \cite{Ewald96} and
references therein, for a discussion of this topic; for explicit examples consider \cite[Ex.~3]{RT-Pic} and the
canonical ambient toric variety presented in \cite[Ex.~3]{R-wMDS}.

In the following, a completion not increasing the rank will be called \emph{sharp}. Although  a sharp
completion of a toric variety does not exist in general,  Hu and Keel showed that the canonical ambient toric variety
$W$, of a MDS $X$, always admits sharp completions, which are even projective, one for each Mori chamber contained in
$\Nef(W)\cong\Nef(X)$ \cite[Prop.~2.11]{Hu-Keel}. Unfortunately this is no more the case for a general wMDS: a
counterexample exhibiting a wMDS whose canonical ambient toric variety does not admit any sharp completion is given in
\cite[Ex.~3]{R-wMDS}.

Theorem~3 in \cite{R-wMDS} characterizes those weak Mori dream spaces $X$ whose ca\-no\-ni\-cal ambient toric
variety $W$ admits a sharp completion $Z$, as those admitting a \emph{filling cell} inside the nef cone $\Nef(X)$: a
filling cell is a cone of the secondary fan of $X$ arising as the common intersection of all the cones of a saturated
bunch of cones containing the bunch of cones associated with $W$ and giving rise to the nef cone of a complete toric
variety \cite[Def.~16]{R-wMDS}.

\begin{definition}[Fillable wMDS]\label{def:fillable} A wMDS $X$ is called \emph{fillable} if $\Nef(X)$ contains a
filling cell $\g$.
\end{definition}

\begin{theorem}[see Thm.~3 in \cite{R-wMDS}]\label{thm:fillable}
  A wMDS $X$ with canonical ambient toric variety $W$ is fillable if and only if there exists a sharp completion
  $W\hookrightarrow Z$. In particular, if $X$ is complete then the induced closed embedding $X\hookrightarrow Z$ is
  neat.
\end{theorem}

\subsection{The canonical 1-covering of a wMDS}
Let $X$ be a wMDS and consider:
\begin{itemize}
  \item its canonical toric embedding $i:X\hookrightarrow W(\Si)$, as given in Theorem~\ref{thm:canonic-emb},
  \item a toric completion $\iota:W\hookrightarrow Z(\g,\Si')$ of $W$, if existing, as given in
      Theorem~\ref{thm:fillable}, and corresponding to the choice of a filling cell
      $$\g\subseteq\Nef(X)\cong\Nef(W)$$
      arising from a \emph{filling fan} $\Si'$ of $\Si$, that is $\Si'\in\SF(V)$ and $\Si\subseteq\Si'$, being $V$ a
      fan matrix of $W$ (and $Z$).
\end{itemize}
Notice that both $W$ and its completion $Z$ are non-degenerate toric varieties. Then Theorem~\ref{thm:QUOT+}
guarantees the existence of universal 1-coverings $\vf:\widetilde{W}\twoheadrightarrow W$ and
$\psi:\widetilde{Z}\twoheadrightarrow Z$.

\begin{remark}\label{rem:covering}
  Since the fan $\Si'$ of $Z$ is a filling fan of the fan $\Si$ of $W$, recalling the construction (\ref{Sigmatilde})
  of the covering fans $\widetilde{\Si}'$ of $\widetilde{Z}$ and $\widetilde{\Si}$ of $\widetilde{W}$, one immediately
  concludes that $\widetilde{\Si}'$ is  a filling fan of $\widetilde{\Si}$, that is $\widetilde{Z}$ is a completion of
  $\widetilde{W}$, giving rise to the following commutative diagram
  \begin{equation}\label{compl-diag}
          \xymatrix{&\widetilde{W}
          \ar@{->>}[d]^-{\vf}\ar@{^(->}[r]^-{\widetilde{\iota}}&
          \widetilde{Z}\ar@{->>}[d]^-{\psi}\\
                    X\ar@{^(->}[r]^-i&W\ar@{^(->}[r]^-{\iota}&Z}
        \end{equation}
  Moreover:
  \begin{enumerate}
    \item $\Cl(\widetilde{W})\cong\Cl(\widetilde{Z})$ is free and
        $\rk\Cl(\widetilde{W})=\rk\Cl(W)=\rk\Cl(Z)=\rk\Cl(\widetilde{Z})$\,;
    \item $\Cox(\widetilde{W})\cong\Cox(W)\cong\K[\X]\cong\Cox(Z)\cong\Cox(\widetilde{Z})$\,, where the left and
        right isomorphisms are $\K$-algebras isomorphisms and not isomorphisms of graded algebras; in fact
        $\Cox(\widetilde{W})$ and $\Cox(\widetilde{Z})$ are graded on $\Cl(\widetilde{W})\cong\Cl(\widetilde{Z})$,
        while $\Cox(W)$ and $\Cox(Z)$ are graded on $\Cl(W)\cong\Cl(Z)$\,;
    \item $\widetilde{W}$ and $\widetilde{Z}$ are 1-connected, hence they are simply connected by
        Proposition~\ref{prop:1-connesso&simplyconn}.
  \end{enumerate}
\end{remark}

We are now in a position to present and prove the following result.

\begin{theorem}\label{thm:notors1-cov}
  A wMDS $X$ admit a canonical 1-covering $\phi:\widetilde{X}\twoheadrightarrow X$ and a canonical closed embedding
  $\widetilde{i}:\widetilde{X}\hookrightarrow\widetilde{W}$ into the universal 1-covering $\widetilde{W}$ of $W$. They
  fit into the following commutative diagram
        \begin{equation}\label{compl-diag-}
          \xymatrix{\widetilde{X}\ar@{->>}^\phi[d]\ar@{^(->}[r]^-{\widetilde{i}}&\widetilde{W}
          \ar@{->>}[d]^-\vf\\
                    X\ar@{^(->}[r]^-i&W}
        \end{equation}
Morover, the following facts are equivalent:
\begin{enumerate}
  \item $\widetilde{i}$ is neat,
  \item $\Cl(\widetilde{X})$ is free and $\rk(\Cl(\widetilde{X}))=\rk(\Cl(X))$\,,
  \item $\widetilde{X}$ is a wMDS and $\Cox(X)\cong\Cox(\widetilde{X})$ are isomorphic as $\K$-algebras, differing
      from each other only by their gradation over $\Cl(X)$ and $\Cl(\widetilde{X})$, respectively.
\end{enumerate}
    Finally, if $X$ is fillable, there is an open embedding $\widetilde{\iota}:\widetilde{W}\hookrightarrow
    \widetilde{Z}$ into the universal 1-covering $\psi:\widetilde{Z}\twoheadrightarrow Z$, completing diagram
    (\ref{compl-diag-}) as follows
        \begin{equation}\label{compl-diag+}
          \xymatrix{\widetilde{X}\ar@{->>}^\phi[d]\ar@{^(->}[r]^-{\widetilde{i}}&\widetilde{W}
          \ar@{->>}[d]^-{\vf}\ar@{^(->}[r]^-{\widetilde{\iota}}&
          \widetilde{Z}\ar@{->>}[d]^-{\psi}\\
                    X\ar@{^(->}[r]^-i&W\ar@{^(->}[r]^-{\iota}&Z}
        \end{equation}
\end{theorem}

\begin{definition}
  In the same notation of Theorem~\ref{thm:notors1-cov}, $\phi:\widetilde{X}\twoheadrightarrow X$ is called the
  \emph{canonical 1-covering} of $X$ and we say that $\widetilde{X}$ is a \emph{torsion-free, rank-preserving,
  1-covering wMDS} of $X$ when the equivalent conditions (1), (2), (3) hold.
\end{definition}

\begin{proof}[Proof of Theorem~\ref{thm:notors1-cov}]
  Given the universal 1-covering $\vf:\widetilde{W}\twoheadrightarrow W$, we get the following short exact sequence of
  abelian groups, asso\-cia\-ted with the canonical torsion subgroup $\Tors(\Cl(W))\leq\Cl(W)$
  \begin{equation*}
    \xymatrix{0\ar[r]&\Tors(\Cl(W))\ \ar@{^(->}[r]&\Cl(W)\ar[r]^-{\vf^*}&\Cl(\widetilde{W})\ar[r]&0}
  \end{equation*}
  Since $\K^*$ is reductive, dualizing over $\K^*$ gives the short exact sequence
  \begin{equation*}
    \xymatrix{1\rightarrow\Hom(\Cl(\widetilde{W}),\K^*)\
    \ar@{^(->}[r]&\Hom(\Cl(W),\K^*)\ar[r]^-{\vf^*}&\Hom(\Tors(\Cl(W)),\K^*)\rightarrow1}
  \end{equation*}
  Since $\Cl(\widetilde{W})$ is free, $H:=\Hom(\Cl(\widetilde{W}),\K^*)$ turns out to be a full subtorus of the
  quasi-torus $G=\Hom(\Cl(W),\K^*)\cong\Hom(\Cl(X),\K^*)$, giving rise to the finite quotient
\begin{equation*}
  \boldsymbol\mu:=\Hom(\Tors(\Cl(W)),\K^*)\cong G/H
\end{equation*}
By item (2) in Remark~\ref{rem:covering}, one has
\begin{equation*}
  \overline{W}=\Spec(\Cox(W))\cong\Spec\K[\X]\cong\K^m\cong\Spec(\Cox(\widetilde{W}))=\overline{\widetilde{W}}
\end{equation*}
where $m=|\X|$.
Under this identification of Cox rings and total coordinate spaces, also irrelevant ideals and loci of $W$ and
$\widetilde{W}$  coincide, by definition (\ref{Sigmatilde}) of the fan $\widetilde{\Si}$. Recalling diagram
(\ref{embediag}), one then has the following quotient description of the 1-covering
$\vf:\widetilde{W}\twoheadrightarrow W$
\begin{equation*}
  \xymatrix{\widetilde{W}\cong j_W(\widehat{W})/H\ar@{->>}[r]^-\vf_-{/\boldsymbol\mu}&j_W(\widehat{W})/G\cong W}
\end{equation*}
and of the canonical toric embedding
\begin{equation*}
 \xymatrix{X\cong \left. j_W\circ\widehat{i}\left(\widehat{X}\right)\right/G\,\ar@{^(->}[r]^-i&
 j_W(\widehat{W})/G\cong W}
\end{equation*}
Define
\begin{equation}\label{tf1-covering}
 \left. \widetilde{X}:=j_W\circ\widehat{i}\,(\widehat{X})\right/H
\end{equation}
This comes with an associated closed embedding $\widetilde{X}\stackrel{\widetilde{i}}{\hookrightarrow}\widetilde{W}$,
equivariant \wrt the $H$-action, and the following commutative diagram
\begin{equation}\label{digramma-emb-cov}
  \xymatrix{\widetilde{X}= \left.
  j_W\circ\widehat{i}\left(\widehat{X}\right)\right/H\,\ar@{^(->}[r]^-{\widetilde{i}}\ar@{->>}[d]_-{\phi}^-{/\boldsymbol\mu}&
  j_W(\widehat{W})/H\cong \widetilde{W}\ar@{->>}[d]_-{\vf}^-{/\boldsymbol\mu}\\
 X\cong \left. j_W\circ\widehat{i}\left(\widehat{X}\right)\right/G\,\ar@{^(->}[r]^-i&
 j_W(\widehat{W})/G\cong W}
\end{equation}
which is precisely the commutative diagram (\ref{compl-diag-}).
Let us show that $\phi:\widetilde{X}\twoheadrightarrow X$ is a 1-covering. In fact
$\vf:\widetilde{W}\twoheadrightarrow W$ is a toric 1-covering and $W$ is non-degenerate. Since $\vf$ is unramified in
codimension 1, Theorem~\ref{thm:1-cov-torico} implies that
\begin{equation*}
  \Br(\vf)\subseteq R:=\bigcup_{1\leq i<j\leq m}D_i\cap D_j
\end{equation*}
Proposition~\ref{prop:neat} shows that $i$ is a neat closed embedding. Then $\Br(\phi)\subseteq X\cap R$ still has
codimension greater than 1 in $X$.

\noindent Notice now that
\begin{equation*}
  \forall\,j=1,\ldots,m\quad \phi^{-1}(i^{-1}(D_j)=(i\circ\phi)^{-1}(D_j)= (\vf\circ
  \widetilde{i})^{-1}(D_j)=\widetilde{i}^{-1}(\vf^{-1}(D_j)
\end{equation*}
Since $i$ is a neat embedding and $\phi$ is a 1-covering, then $\{\phi^{-1}(i^{-1}(D_j)\}_{j=1}^m$ is a set of
pairwise distinct hypersurfaces of $\widetilde{X}$. On the other hand, $\{\vf^{-1}(D_j)\}_{j=1}^m$ is the set of torus
invariant prime divisors of $\widetilde{W}$. Then the closed toric embedding $\widetilde{i}$ satisfies hypotheses of
Proposition~\ref{prop:pullback}, so giving a well defined pull back homomorphism
$\widetilde{i}^*:\Cl(\widetilde{W})\to\Cl(\widetilde{X})$\,. Consider the following commutative diagram of group
homomorphisms
\begin{equation*}
  \xymatrix{\Cl(W)\ar[d]^-{\vf^*}_-\cong\ar[r]^-{i^*}_-\cong&\Cl(X)\ar[d]^-{\phi^*}\\
            \Cl(\widetilde{W})\ar[r]^-{\widetilde{i}^*}&\Cl(\widetilde{X})}
\end{equation*}
being the pull back $\phi^*:\Cl(X)\to\Cl(\widetilde{X})$ well defined by (\ref{Weilpullback}) in
\S~\ref{ssez:pullback}. Assume the following fact, whose proof is postponed.

\begin{lemma}\label{lem:ker}
  $\ker\phi^*=\Tors(\Cl(X))$
\end{lemma}

Therefore $\rk(\im\phi^*)=\rk(\Cl(X))=\rk(\Cl(W))=\rk(\Cl(\widetilde{W}))$, meaning that $\widetilde{i}$ is neat if an
only if $\phi^*$ is surjective, that is if and only if $\Cl(\widetilde{X})$ is free and
$\rk(\Cl(\widetilde{X}))=\rk(\Cl(X))$, proving that $(1)\,\Leftrightarrow\,(2)$.

\noindent To show that $(2)\,\Leftrightarrow\,(3)$, notice that by construction we have the following commutative
diagram
\begin{equation*}
\xymatrix{\widehat{X}\ar@{->>}[dd]^-{p_X}\ar@{->>}[rd]^-{p_{\widetilde{X}}}\ar@{^(->}[rr]^-{\widehat{i}}
              &&\widehat{W}\ar@{->>}[rd]^-{p_{\widetilde{W}}}\ar@{->>}[dd]^<<<<<<<{p_W}&\\
              &\widetilde{X}\ar@{->>}[ld]^-\phi\ar@{^(->}[rr]^<<<<<<{\widetilde{i}}&&\widetilde{W}\ar@{->>}[ld]^-\vf\\
              X\ar@{^(->}[rr]^-{i}&&W&}
\end{equation*}
Define $\widetilde{\cox}:=(p_{\widetilde{X}})_*\cO_{\widehat{X}}$\,.
Recall that the canonical morphism $p_X$ of the relative spectrum construction give the following isomorphism
\begin{equation*}
  \cox\cong (p_X)_*\cO_{\widehat{X}}=\phi_*\left(
  (p_{\widetilde{X}})_*\cO_{\widehat{X}}\right)=\phi_*\,\widetilde{\cox}
\end{equation*}
Passing to global sections and observing that $\phi^{-1}(X)=\widetilde{X}$, we get that
\begin{equation*}
  \Cox(X)=\Ga(X,\cox)\cong\Ga(\widetilde{X},\widetilde{\cox})
\end{equation*}
This is not an isomorphism of graded algebras, but it suffices to prove that $\Ga(\widetilde{X},\widetilde{\cox})$ is
a finitely generated algebra.

\noindent For what concerning their gradations, notice that
\begin{equation*}
  \cox=\bigoplus_{\d\in\Cl(X)}\cox_\d\ \cong\ \phi_*\,\widetilde{\cox}=
  \bigoplus_{\eta\in\im\phi^*}\phi_*\,\widetilde{\cox}_\eta\cong
  \bigoplus_{\eta\in\im\phi^*}\left(\bigoplus_{\d\in(\phi^*)^{-1}(\eta)} \cox_\d\right)
\end{equation*}
Call $\widetilde{X}'$ the wMDS admitting Cox sheaf and class group given by $\widetilde{\cox}$ and
$\Cl(\widetilde{X}')=\im\phi^*$, respectively. Applying Theorem~\ref{thm:canonic-emb} and Proposition~\ref{prop:neat}
to $\widetilde{X}$ and $\widetilde{X}'$, by replacing the quasi-torus action of $G$ with the torus action of $H$ and
$H':=\Hom(\im\phi^*,\K^*)$, respectively, one gets
\begin{equation*}
   \left(\widetilde{X}'=\Spec_{\widetilde{X}}(\widetilde{\cox})/H'\right)\ \cong\ \left(\widehat{X}/H=
   \widetilde{X}\right)\\ \Longleftrightarrow\ \im\phi^*=\Cl(\widetilde{X})
\end{equation*}
For what concerning the last part of the statement, notice that $X$ is fillable if and only if $\widetilde{X}$ is
fillable. In particular, recalling diagram (\ref{compl-diag}), the commutative diagram
(\ref{digramma-emb-cov}) extends to give the following one
\begin{equation*}
  \xymatrix{\widetilde{X}= \left.
  j_W\circ\widehat{i}\left(\widehat{X}\right)\right/H\,\ar@{^(->}[r]^-{\widetilde{i}}\ar@{->>}[d]_-{\phi}^-{/\boldsymbol\mu}&
  j_W(\widehat{W})/H\cong \widetilde{W}\ar@{^(->}[r]^-{\widetilde{\iota}}\ar@{->>}[d]_-{\vf}^-{/\boldsymbol\mu}&
  j_Z(\widehat{Z})/H\cong \widetilde{Z}\ar@{->>}[d]_-{\psi}^-{/\boldsymbol\mu}\\
 X\cong \left. j_W\circ\widehat{i}\left(\widehat{X}\right)\right/G\,\ar@{^(->}[r]^-i&
 j_W(\widehat{W})/G\cong W\ar@{^(->}[r]^-{\iota}&j_Z(\widehat{Z})/G\cong Z}
\end{equation*}
which is precisely the commutative diagram (\ref{compl-diag+}).  \end{proof}

\begin{proof}[Proof of Lemma \ref{lem:ker}] \hfill\newline
$\Tors(\Cl(X))\subseteq \ker\phi^*$\,. In fact, if $\d\in\Tors(\Cl(X))$ then $(i^*)^{-1}(\d)\in\Tors(\Cl(W))$.
Therefore $\vf^*((i^*)^{-1}(\d))=0$, so giving $\phi^*(\d)=\widetilde{i}^*\circ\vf^*\circ (i^*)^{-1}(\d)=0$\,.

\halfline\noindent $\ker\phi^*\subseteq\Tors(\Cl(X))$\,. Consider $\d\in\Cl(X)$ such that $\phi^*(\d)=0$. Then, for
every $D\in d_K^{-1}(\d)$ the divisor $\phi^\#(D)=\phi^{-1}(D)$ is principal. In particular it is an invariant divisor
\wrt the action of $\Bmu$, meaning that $\phi^\#(D)=(f)$ for some $\Bmu$-homogeneous function
$f\in\K(\widetilde{X})^*$. Consider the $|\Bmu|$-power $q:\K\to\K$, such that $q(z)=z^{|\Bmu|}$, and define
$\overline{f}\in\K(X)^*$ by setting
\begin{equation*}
  \forall\,x\in X\quad \overline{f}(x):=q(f(y))\ \text{for some}\ y\in\phi^{-1}(x)
\end{equation*}
$\overline{f}$ is well defined because $f$ $\Bmu$-homogeneous gives
$$\forall\,\zeta\in\Bmu\quad q(f(\zeta\cdot y))=q(f(y))$$
Notice that $|\Bmu|D=\phi(\phi^\#(D))=(\overline{f})\sim 0$\,, so giving that $D\in \Tors(\Cl(X))$\,.
\end{proof}

 \begin{remark}
   Notice that the 1-covering $\phi:\widetilde{X}\twoheadrightarrow X$ is \emph{canonical}, in the sense that it does
   not depend on the choice of the set of generators $\X$. In fact, for a different choice $\X'$,  let
   $i':X\hookrightarrow W'$ be the $\X'$-canonical toric embedding. By Proposition~ \ref{prop:neat}
   $$G:=\Hom(\Cl(X),\K^*)\cong\Hom(\Cl(W),\K^*)\cong\Hom(\Cl(W'),\K^*)$$
     Then a free part of $G$ is isomorphic to $H$, that is
         \begin{equation*}
           \Hom(\Cl(\widetilde{W}),\K^*)\cong H\cong\Hom(\Cl(\widetilde{W}'),\K^*)
         \end{equation*}
         and the same holds for the torsion subgroup
         \begin{equation*}
           \Hom(\Tors(\Cl(W)),\K^*)\cong\mu\cong\Hom(\Tors(\Cl(W')),\K^*)\,.
         \end{equation*}
         On the other hand
     $\widehat{X}=\Spec_X\left(\cox\right)\cong\widehat{X}'$\,.
   Therefore the 1-covering
   $$\xymatrix{\widetilde{X}=\left.j_W\circ\widehat{i}\left(\widehat{X}\right)\right/H\cong
   \left.j_{W'}\circ\widehat{i}'\left(\widehat{X}'\right)\right/H=\widetilde{X}'\ar@{->>}[r]^-\phi_-{/\mu}&X}$$
   is canonically fixed, up to isomorphisms.
    \end{remark}

\subsection{When is the canonical embedding of the canonical 1-covering a neat embedding?}\label{ssez:neat}

Given a wMDS $X$ with canonical toric embedding $i:X\hookrightarrow W$, let $\phi:\widetilde{X}\twoheadrightarrow X$
be the canonical 1-covering, constructed in Theorem~\ref{thm:notors1-cov}, and
$\widetilde{i}:\widetilde{X}\hookrightarrow \widetilde{W}$ be its canonical closed toric embedding giving rise to the
commutative diagram (\ref{compl-diag-}). Keeping in mind the equivalent conditions (1), (2), (3) in the statement of
Theorem~\ref{thm:notors1-cov}, being neat for $\widetilde{i}$ is a sort of extension to $\Q$-factorial varieties of
the Grothendieck-Lefschetz theorem \cite[Exp.~XI]{SGA2}, for the class group morphism
$\widetilde{i}^*:\Cl(\widetilde{W})\to\Cl(\widetilde{X})$. Following \cite{Jow}, \cite{AL} and \cite{RS}, we can
obtain sufficient conditions to get neatness of $\widetilde{i}$. At this purpose we need to introduce the following

\begin{definition}
  A $\Q$-factorial toric variety $W=W(\Si)$ (or equivalently its simplicial fan $\Si$) is called \emph{$k$-neighborly}
  if for any $k$ rays in $\Si(1)$ the convex cone they span is in $\Si(k)$. Equivalently, by Gale duality, this means
  that
  \begin{equation}\label{k-neighborly}
    \Nef(X)\subseteq\bigcap_{1\leq i_1<\cdots<i_k\leq |\Si(1)|} \left\langle Q^{\{i_1,\ldots,i_k\}}\right\rangle
  \end{equation}
\end{definition}

The following characterization of a $k$-neighborly toric variety follows by the inclusion (\ref{k-neighborly}),
recalling the natural correspondence between the bunch of cones of $W$ and the generators of its irrelevant ideal $\I
rr(W)$. See also \cite[Prop.~10, Rmk.~11]{Jow} for further details.

\begin{proposition}\label{prop:neighborly}
  A $\Q$-factorial toric variety $W$ is $k$-neigh\-bor\-ly if and only if the irrelevant locus
  $\widetilde{B}\subseteq\overline{W}$ has codimension $\codim_{\overline{W}}\widetilde{B}>k\,.$
\end{proposition}

We are now in a position of giving the following sufficient conditions for the neatness of $\widetilde{i}$.

\begin{proposition}\label{prop:s.c.neat}
Let $X,\widetilde{X},W,\widetilde{W}$ as above, then the canonical closed embedding
$\widetilde{i}:\widetilde{X}\hookrightarrow\widetilde{W}$ is neat if one of the following happens:
\begin{enumerate}
\item if $\widetilde{X}$ is a smooth complete intersection of codimension $l$ in $\widetilde{W}$ and the latter is a
    smooth, projective, $(1+l)$-neighborly toric variety,  with $\dim(\widetilde{W})\geq 3+l$;
\item if $X$ is a complete intersection of codimension $l$ in $W$ and the irrelevant locus $B_X\subset \overline{X}$
    has codimension $\geq 1+l$;
\item if $\widetilde{X}\in|D|$ is a general element, with $D$ an ample divisor of $\widetilde{W}$ and the latter is
    projective with $\dim(\widetilde{W})\geq 4$.
\end{enumerate}
\end{proposition}

\begin{proof}
  (1) is an iterated application of \cite[Thm.~6]{Jow}, keeping in mind the equivalence established by
  Proposition~\ref{prop:neighborly} and recalling the equivalence (1)\,$\Leftrightarrow$\,(3) in
  Theorem~\ref{thm:notors1-cov}. For (2) notice that by the commutative diagram (\ref{compl-diag-}),  $X$ is a
  complete intersection of codimension $l$ in $W$ if and only if $\widetilde{X}$ is a complete intersection of
  codimension $l$ in $\widetilde{W}$ and
  $    \codim_{\overline{X}}B_X =\codim_{\overline{X}}B_{\widetilde{X}}$.
  Then apply \cite[Thm.~2.1]{AL} and equivalence (1)\,$\Leftrightarrow$\,(2) in Theorem~\ref{thm:notors1-cov} to get
  the neatness of $\widetilde{i}$. Finally (3) is a direct application of \cite[Thm.~1]{RS}, recalling equivalence
  (1)\,$\Leftrightarrow$\,(2) in Theorem~\ref{thm:notors1-cov}.
\end{proof}

\begin{remark}
  Let us notice that non-neatness of the canonical embedding $\widetilde{i}:\widetilde{X}\hookrightarrow\widetilde{W}$ can happen. Consider the example given by a particular Enriques surface admitting the structure of a MDS and explained in the following \S~\ref{ssez:Enriques}. In this case, the canonical 1-covering $\widetilde{X}$ is a $K3$ surface which can never be a MDS, so contradicting condition (3) of Theorem~\ref{thm:notors1-cov}.
\end{remark}

 \subsection{When is the canonical 1-covering the universal one?}\label{ssez:gruppifondamentali}

 Let $X$ be a fillable wMDS and $\widetilde{X}$ be its canonical 1-covering. Let $X\hookrightarrow Z$ and
 $\widetilde{X}\hookrightarrow \widetilde{Z}$ be complete toric embeddings assigned by the choice of a filling chamber
 $\g\subseteq \Nef(W)$, as in Theorem~\ref{thm:notors1-cov}, diagram (\ref{compl-diag+}).
 Corollary~\ref{cor:univ-1-cov} allows us to conclude that
 \begin{itemize}
   \item \emph{$\phi:\widetilde{X}\twoheadrightarrow X$ is the universal 1-covering of $X$ if and only if the open
       subset $\widetilde{X}_{\rm{reg}}$, of regular points of $\widetilde{X}$, is simply connected i.e.
       $\pet(\widetilde{X}_{\rm{reg}},x)=0$ for every regular point $x\in \widetilde{X}_{\rm{reg}}$\,.}
 \end{itemize}
 Notice that $\vf:\widetilde{Z}\twoheadrightarrow Z$ is the universal 1-covering of $Z$, that is
 $\pet(\widetilde{Z}_{\rm{reg}},z)=0$, for every regular point $z\in \widetilde{Z}_{\rm{reg}}$. Therefore asking for
 simply connectedness of $\widetilde{X}_{\rm{reg}}$ translates in a sort of Weak Lefschetz Theorem on the \ét fundamental
 groups of smooth loci in $\widetilde{X}\hookrightarrow \widetilde{Z}$. Clearly we cannot hope this result to hold in general.
  In the following we consider the particular case $\K=\C$ with, in addition, some strong hypotheses on singularities
  of $\widetilde{X}$ and the embedding $\widetilde{X}\hookrightarrow \widetilde{Z}$.

\begin{definition}
Let $X$ be a wMDS and $i:X\hookrightarrow W$ be its canonical toric embedding.
$X$ is called a \emph{complete intersection} if the relations' ideal $I\subset\K[\X]$, such that
$\Cox(X)\cong\K[\X]/I$, is generated by exactly $l:=\codim_W(X)$ polynomials.

Moreover, $X$ is called a \emph{very ample} complete intersection if there exists a sharp completion $Z$ of $W$ and every generator of $I$ defines a very ample divisor in $Z$.

Finally, $X$ is called a \emph{quasi-smooth} complete intersection if the generators of $I$ can be ordered to giving $I=(f_1,f_2,\ldots,f_c)$ in such a way that
\begin{eqnarray*}
  X=Y^c:=\{f_1=\cdots=f_{c-1}=f_c=0\}&\subset& Y^{c-1}:=\{f_1=\cdots=f_{c-1}=0\}\subset\cdots\\
  \cdots&\subset& Y^1:=\{f_1=0\}\subset Y^0:=Z
\end{eqnarray*}
and
\begin{equation*}
  \forall\,i=0,\ldots, c-1\quad Y^i_{\text{\rm reg}}\cap Y^{i+1}\subseteq  Y^{i+1}_{\text{\rm reg}}
\end{equation*}
 \end{definition}

  \begin{definition}[Small $\Q$-factorial modification]\label{def:sQm}
  A birational map $f:X\dashrightarrow Y$, between complete and $\Q$-factorial algebraic varieties, is
  called a \emph{small $\Q$-factorial modification} (s$\Q$m) if it is biregular in codimension 1 i.e. there exist
  Zariski big open subsets $U\subseteq X$ and $V\subseteq Y$ such that $f|_U:U\stackrel{\cong}{\longrightarrow} V$ is
  biregular.
\end{definition}

\begin{remark}
  Notice that a $\Q$-factorial and complete algebraic variety $X$ is a wMDS if and only if there exists a \sqm
  $X\dashrightarrow X'$ such that $X'$ is a MDS \cite[Lemma~1]{R-wMDS}.
\end{remark}

 \begin{theorem}\label{thm:WeakLefschetz}
Assume $\K=\C$ and let $X$ be a complete and fillable wMDS admitting a \sqm
$X\dashrightarrow X'$ to a MDS $X'$ which is a quasi-smooth and very ample complete intersection. Then the canonical torsion free 1-covering
$\widetilde{X}\twoheadrightarrow X$ is the universal 1-covering of $X$.
In particular, a MDS which is a quasi-smooth and very ample complete intersection is simply connected and always admits a universal 1-covering.
 \end{theorem}

 The previous statement is obtained by the following
 result of Goresky and MacPherson

 \begin{theorem}[see \S\,II.1.2 in \cite{G-MacPh}]\label{thm:G-MacP}
   Let $Y$ be the complement of a closed subvariety of a $n$-dimensional complex analytic variety $\overline{Y}$ and
   $j:\overline{Y}\to \P^N$ be a proper embedding. Let $H\subseteq\P^N$ be a hyperplane. Then the homomorphism induced
   by inclusion on the fundamental groups $\pi_1((j|_Y)^{-1}(H))\to\pi_1(Y)$ is an isomorphism for $n\geq 2$.
 \end{theorem}

 This statement is derived from the theorem at the beginning of \S\,II.1.2 in \cite{G-MacPh}, assuming condition (1) immediately following the statement of that theorem, since $j$ is proper, as well as condition (2).

 \begin{proof}[Proof of Thm.\,\ref{thm:WeakLefschetz}]
 The \sqm $X\dashrightarrow X'$ fits into the following 3-dimensional commutative diagram
 \begin{equation*}
   \xymatrix{\widetilde{X}\ar@{^(->}[dr]\ar@{-->}[rr]^-{\widetilde{f}|_{\widetilde{X}}}\ar@{->>}[dd]^-{\phi}&&\widetilde{X}'\ar@{^(->}[dr]\ar@{->>}[dd]^-<<<<<<<{\phi'}&\\
            &\widetilde{Z}\ar@{-->}[rr]^-<<<<<<<{\widetilde{f}}\ar@{->>}[dd]^-<<<<<<<{\vf}&&\widetilde{Z}'\ar@{->>}[dd]^-{\vf'}\\
            X\ar@{-->}[rr]^-<<<<<<{f|_X}\ar@{^(->}[dr]&&X'\ar@{^(->}[dr]&\\
            &Z\ar@{-->}[rr]^-f&&Z'}
 \end{equation*}
 where:
  \begin{itemize}
    \item vertical maps $\phi$ and $\phi'$ are canonical torsion-free 1-coverings,
    \item vertical maps $\vf$ and $\vf'$ are canonical universal 1-coverings,
    \item diagonal maps are complete toric embeddings associated with the choice of a filling cell
        $\g\subseteq\Nef(X)\cong\Nef(W)$,
    \item horizontal maps $f,\widetilde{f},f|_X,\widetilde{f}|_{\widetilde{X}}$ are small $\Q$-factorial
        modifications: in particular $X'$, $\widetilde{X}'$ are MDS and $Z'$, $\widetilde{Z}'$ are projective
        $\Q$-factorial toric varieties and
    \begin{equation}\label{pi1=0}
\pi_1(\widetilde{Z}'_{\rm{reg}})\cong\{1\}
\end{equation}
  \end{itemize}
Let us first assume that $X'$ is a very ample hypersurface in the projective toric variety
$Z'$. Then, by definition (\ref{tf1-covering}), $\widetilde{X}'$ is a very ample hypersurface of $\widetilde{Z}'$, so giving the following projective embedding of $\widetilde{X}'$
 \begin{equation}\label{embedding}
   \xymatrix{\widetilde{X}'\ar@{^(->}[rr]\ar@{^(->}[rd]&&\P^N\\
            &\widetilde{Z}'\ar@{^(->}[ru]^j&}
 \end{equation}
 so that $\widetilde{X}'=j^{-1}(H)$, for a suitable hyperplane $H\subset\P^N$.
 Apply now Theorem \ref{thm:G-MacP} by setting
 $\overline{Y}=\widetilde{Z}'$, $ Y=\widetilde{Z}'_{\rm{reg}}$\,. Quasi-smoothness of $\widetilde{X}'$ implies  that
 \begin{equation*}
 (j|_{\widetilde{Z}'_{\rm{reg}}})^{-1}(H)=\widetilde{Z}'_{\rm{reg}}\cap\widetilde{X}'\subseteq\widetilde{X}'_{\rm{reg}}
\end{equation*}
The latter inclusion induces a covariant surjection on associated fundamental groups (see e.g. \cite[Thm.~12.1.5]{CLS}
and references therein), so giving
\begin{equation*}
  \{1\}\cong\pi_1(\widetilde{Z}'_{\rm{reg}})\cong\pi_1(\widetilde{Z}'_{\rm{reg}}\cap\widetilde{X}')\twoheadrightarrow
  \pi_1(\widetilde{X}'_{\rm{reg}})\quad\Rightarrow\quad\pi_1(\widetilde{X}'_{\rm{reg}})\cong\{1\}
\end{equation*}
by relation (\ref{pi1=0}) and Theorem~\ref{thm:G-MacP}.
The last step is proving that $\pi_1(\widetilde{X}_{\rm{reg}})\cong\{1\}$\,. In fact, since $\widetilde{X}$ and
$\widetilde{X}'$ are normal and related by the \sqm $\widetilde{f}|_{\widetilde{X}}$, then $\widetilde{X}_{\rm{reg}}$ and
$\widetilde{X}'_{\rm{reg}}$ are smooth and biregular in codimension 1. Then Theorem~\ref{thm:excisione} and
Remark~\ref{rem:K=C+finite} give that $\pi_1(\widetilde{X}_{\rm{reg}})\cong\pi_1(\widetilde{X}'_{\rm{reg}})\cong\{1\}$.

Let us now assume $X'$ be a complete intersection of $c\geq 2$ hypersurfaces of $W'$, hence of its completion $Z'$. This
means that $I=(f_1,\ldots,f_c)$ in $\Cox(Z')\cong\C[x_1,\ldots,x_m]$, where $\C[x_1,\ldots,x_m]/I\cong\Cox(X')$. Then
$X'$ is a very ample hypersurface of the complete intersection $Y$ of $W'$ associated with the ideal $(f_1,\ldots,f_{c-1})$ in the
construction given by Theorem~\ref{thm:canonic-emb}. Then by definition (\ref{tf1-covering}), $\widetilde{X}'$  is a very ample
hypersurface of the complete intersection $\widetilde{Y}\subseteq\widetilde{Z}'$.
In particular, $\widetilde{Y}$ is projective and, by induction on $c$, we can assume
\begin{equation*}
  \pi_1(\widetilde{Y}_{\rm{reg}})\cong\{1\}
\end{equation*}
Diagram (\ref{embedding}) can be replaced by the following projective embedding of $\widetilde{X}'$
\begin{equation*}
   \xymatrix{\widetilde{X}'\ar@{^(->}[rr]\ar@{^(->}[rd]&&\P^N\\
            &\widetilde{Y}\ar@{^(->}[ru]^j&}
 \end{equation*}
 so that $\widetilde{X}'=j^{-1}(H)$, for a suitable hyperplane $H\subset\P^N$.
 Apply now Theorem \ref{thm:G-MacP} by setting
 $\overline{Y}=\widetilde{Y}'$, $ Y=\widetilde{Y}'_{\rm{reg}}$\,. Quasi-smoothness of $\widetilde{X}'$ implies  that
 \begin{equation*}
 (j|_{\widetilde{Y}'_{\rm{reg}}})^{-1}(H)=\widetilde{Y}'_{\rm{reg}}\cap\widetilde{X}'\subseteq\widetilde{X}'_{\rm{reg}}
\end{equation*}
Therefore
\begin{equation*}
  \{1\}\cong\pi_1(\widetilde{Y}'_{\rm{reg}})\cong\pi_1(\widetilde{Y}'_{\rm{reg}}\cap\widetilde{X}')\twoheadrightarrow
  \pi_1(\widetilde{X}'_{\rm{reg}})\quad\Rightarrow\quad\pi_1(\widetilde{X}'_{\rm{reg}})\cong\{1\}
\end{equation*}
The last step, proving that $\pi_1(\widetilde{X}_{\rm{reg}})\cong\pi_1(\widetilde{X}'_{\rm{reg}})\cong\{1\}$, proceeds exactly
as in case $c=1$.

If $X$ is a MDS which is a quasi-smooth and very ample complete intersection, one can run the previous argument by taking $f$ as the
identity. Finally, the simply connectedness of the MDS $X$ is proved by setting $Y=\overline{Y}$, that is, by assuming
the closed subvariety in the statement of Theorem~\ref{thm:G-MacP}, as empty. Then apply the same inductive argument
by starting with $Y=Z$ and recalling that a complete toric variety is always simply connected, by
Corollary~\ref{cor:Grothendieck}.
\end{proof}

\begin{remark}
  Theorem~\ref{thm:WeakLefschetz} can be certainly generalized to admitting some further singularity for
  either $X$ or $X'$: in fact Goresky-MacPherson results are more general than Theorem~\ref{thm:G-MacP}, which
  presents a statement adapted to the case here considered. However, any such ge\-ne\-ra\-lization strongly depends on
  the kind of admitted singularities for $X$ and needs a careful application of deep and more general results due to
  Goresky-MacPherson and Hamm-L\^{e} (see \cite{G-MacPh}, \cite{Hamm-Le}).

  The reader is referred to \S~\ref{ex:universal} for an example of a MDS satisfying the hypotheses of Theorem~\ref{thm:WeakLefschetz}: it could be of some interest understanding how those hypotheses can be checked.
\end{remark}

\section{Examples and further applications}\label{sez:esempi}

This section is devoted to present examples of Mori dream spaces and their canonical 1-covering. The first example considers a MDS whose canonical 1-covering admits a neat
embedding in its canonical ambient toric variety, but unfortunately we are not able to conclude if it is the universal 1-covering, as hypotheses of Theorem~\ref{thm:WeakLefschetz} are not completely satisfied. On the contrary the second example presents a MDS whose canonical 1-covering admits a neat
embedding in its canonical ambient toric variety and it is also the universal 1-covering, as obtained by checking hypotheses of Theorem~\ref{thm:WeakLefschetz}. Finally, the third example presents a class of Mori Dream surfaces, namely those particular Enriques surfaces which are MDS, whose canonical 1-covering does not even admit a neat
embedding in a toric variety, as it is no more a MDS.

\subsection{An example by Hausen and Keicher}\label{ssez:HK}

Example here presented is obtained by considering, up to isomorphism, the Cox ring studied in
\cite[Ex.~2.1]{MDSpackage} and also listed in the \emph{Cox ring database} \cite{CRdb}, where it is reported as the id
no.~97.

Consider the grading map $d_K:K=\Z^8\twoheadrightarrow\Z^3\oplus\Z/2\Z$, whose free part is represented by the weight
matrix
\begin{equation*}
  Q=\left( \begin {array}{cccccccc} 2&1&0&2&0&2&1&0\\ 1&1&1&1&1&1&1&1\\ 0&0&0&1&1&2&2&2\end {array} \right) =\left(
             \begin{array}{ccc}
               \q_1 & \cdots & \q_8 \\
             \end{array}
           \right)
\end{equation*}
and whose torsion part is represented by the torsion matrix
\begin{equation*}
  \mathcal{T}=\left( \begin {array}{cccccccc}
  \overline{0}&\overline{0}&\overline{0}&\overline{0}&\overline{1}&\overline{1}&\overline{1}&\overline{1} \end {array}
  \right)
\end{equation*}
Then, consider the quotient algebra
$$R=\K[x_1,\ldots,x_8]/(x_1x_8+x_2x_7+x_3x_6+x_4x_5)$$
graded by $d_K$. This is consistent since the relation defining $R$ is homogeneous \wrt such a grading. Moreover $R$
turns out to be a Cox ring with $\X:=\{\overline{x}_1,\ldots,\overline{x}_8\}$ giving a Cox basis of $R$. Then
$\overline{X}:=\Spec(R)\subseteq \Spec\K[\x]=:\overline{W}$ defines the total coordinate space of a wMDS
$X:=\widehat{X}/G$ and its canonical ambient toric variety $W:=\widehat{W}/G$, where
\begin{eqnarray*}
  \widehat{X}&=&\overline{X}\backslash B_X\quad \text{being}\quad B_X=\mathcal{V}(\irr(X))\\
  \widehat{W}&=&\overline{W}\backslash \widetilde{B}\quad \text{being}\quad
  \widetilde{B}=\mathcal{V}(\widetilde{\irr})\\
  \irr(X)&=&\left(\begin{array}{c}
  \overline{x}_1\overline{x}_3\overline{x}_7,\overline{x}_1\overline{x}_5\overline{x}_6,
  \overline{x}_1\overline{x}_5\overline{x}_7,\overline{x}_2\overline{x}_4\overline{x}_8,
  \overline{x}_2\overline{x}_5\overline{x}_6,\overline{x}_2\overline{x}_6\overline{x}_8,
  \overline{x}_3\overline{x}_4\overline{x}_7,\overline{x}_3\overline{x}_4\overline{x}_8\\
  \overline{x}_1\overline{x}_2\overline{x}_7\overline{x}_8,
  \overline{x}_1\overline{x}_3\overline{x}_6\overline{x}_8,
  \overline{x}_1\overline{x}_4\overline{x}_5\overline{x}_8,
    \overline{x}_2\overline{x}_3\overline{x}_6\overline{x}_7,
  \overline{x}_2\overline{x}_4\overline{x}_5\overline{x}_7,
  \overline{x}_3\overline{x}_4\overline{x}_5\overline{x}_6
  \end{array}\right)\\
  \widetilde{\irr}&=&\left(\begin{array}{c}
  {x}_1{x}_3{x}_7,{x}_1{x}_5{x}_6,
  {x}_1{x}_5{x}_7,{x}_2{x}_4{x}_8,
  {x}_2{x}_5{x}_6,{x}_2{x}_6{x}_8,
  {x}_3{x}_4{x}_7,{x}_3{x}_4{x}_8\\
  {x}_1{x}_2{x}_7{x}_8,
  {x}_1{x}_3{x}_6{x}_8,
  {x}_1{x}_4{x}_5{x}_8,
    {x}_2{x}_3{x}_6{x}_7,
  {x}_2{x}_4{x}_5{x}_7,
  {x}_3{x}_4{x}_5{x}_6
  \end{array}\right)\\
  G&=&\Hom(\Cl(W),\K^*)\cong\Hom(\Z^3\oplus\Z/2\Z,\K^*)
\end{eqnarray*}
A Gale dual matrix of $Q$ is given by the following $CF$-matrix
\begin{equation*}
  \widetilde{V}=\left( \begin {array}{cccccccc} 1&0&0&0&-2&0&-2&3\\ 0&1&0&0&-2&0&-1&2\\ 0&0&1&0&-2&0&0&1\\
  0&0&0&1&-1&0&-2&2\\ 0&0&0&0&0&1&-2&1\end {array} \right) =\left(
             \begin{array}{ccc}
               \widetilde{\v}_1 & \cdots & \widetilde{\v}_8 \\
             \end{array}
           \right)
\end{equation*}
Notice that $Q$ is a Gale dual matrix of both $\widetilde{V}$ and the following $F$-matrix
\begin{equation*}
  V=\left( \begin {array}{cccccccc} 1&0&0&1&-3&0&-4&5\\ 0&1&0&1&-3&0&-3&4\\ 0&0&1&1&-3&0&-2&3\\ 0&0&0&2&-2&0&-4&4\\
  0&0&0&0&0&1&-2&1\end {array} \right)=\left(
             \begin{array}{ccc}
               \v_1 & \cdots & \v_8 \\
             \end{array}
           \right)
\end{equation*}
Moreover, it turns out that $\mathcal{T}\cdot V^T=\mathbf{0}$. Then $V$ is a fan matrix of $W$, while $\widetilde{V}$
is a fan matrix of the universal 1-covering $\widetilde{W}$ of $W$. In particular,
\begin{equation*}
  \widetilde{W}= \widehat{W}/H\quad\text{where}\quad H:=\Hom(\Cl(\widetilde{W}),\K^*)\cong\Hom(\Z^3,\K^*)
\end{equation*}
The canonical torsion free 1-covering $\widetilde{X}$ of $X$ is the given by $\widetilde{X}= \widehat{X}/H$. It is a
MDS whose canonical ambient toric variety is given by $\widetilde{W}$. In particular,
$\widetilde{i}:\widetilde{X}\hookrightarrow\widetilde{W}$ is a neat embedding. Notice that $X$ is quasi-smooth but it is not very ample: in fact it is not even a Cartier divisor of $W$ and, then, of any sharp completion $Z$ of $W$. Then, when $\K=\C$, hypotheses of Theorem~\ref{thm:WeakLefschetz} are not satisfied.

\subsection{An example satisfying hypotheses of Theorem~\ref{thm:WeakLefschetz}}\label{ex:universal}
Consider the grading map $d_K:K=\Z^5\twoheadrightarrow\Z^2\oplus\Z/3\Z$, whose free part is represented by the weight
matrix
\begin{equation*}
  Q=\left( \begin {array}{ccccc} 1&1&1&0&2\\ 0&2&1&1&1\end {array} \right) =\left(
             \begin{array}{ccc}
               \q_1 & \cdots & \q_5 \\
             \end{array}
           \right)
\end{equation*}
and whose torsion part is represented by the torsion matrix
\begin{equation*}
  \mathcal{T}=\left( \begin {array}{ccccc}
  \overline{1}&\overline{0}&\overline{1}&\overline{0}&\overline{0} \end {array}
  \right)
\end{equation*}
Then, consider the quotient algebra
$$R=\C[x_1,\ldots,x_5]/(x_1^3x_2^9+x_1^{12}x_4^{18}+x_2^6x_3^6+x_3^{12}x_4^6+x_4^{12}x_5^6)$$
graded by $d_K$. This is consistent since the relation defining $R$ is homogeneous \wrt such a grading. Moreover $R$
turns out to be a Cox ring with $\X:=\{\overline{x}_1,\ldots,\overline{x}_5\}$ giving a Cox basis of $R$. Then
$\overline{X}:=\Spec(R)\subseteq \Spec\C[\x]=:\overline{W}$ defines the total coordinate space of a wMDS
$X:=\widehat{X}/G$ and its canonical ambient toric variety $W:=\widehat{W}/G$, where
\begin{eqnarray*}
  \widehat{X}&=&\overline{X}\backslash B_X\quad \text{being}\quad B_X=\mathcal{V}(\irr(X))\\
  \widehat{W}&=&\overline{W}\backslash \widetilde{B}\quad \text{being}\quad
  \widetilde{B}=\mathcal{V}(\widetilde{\irr})\\
  \irr(X)&=&\left(\begin{array}{c}
  \overline{x}_1\overline{x}_2,\overline{x}_1\overline{x}_4,
  \overline{x}_2\overline{x}_3,\overline{x}_2\overline{x}_5,\overline{x}_3\overline{x}_4,\overline{x}_4\overline{x}_5
  \end{array}\right)\\
  \widetilde{\irr}&=&\left(\begin{array}{c}
  {x}_1{x}_2,{x}_1{x}_4,
  {x}_2{x}_3,{x}_2{x}_5,{x}_3 x_4,{x}_4{x}_5
  \end{array}\right)\\
  G&=&\Hom(\Cl(W),\C^*)\cong\Hom(\Z^2\oplus\Z/3\Z,\C^*)
\end{eqnarray*}
A Gale dual matrix of $Q$ is given by the following $CF$-matrix
\begin{equation*}
  \widetilde{V}=\left( \begin {array}{ccccc}  1&0&1&0&-1\\ 0&1&1&-2&-1\\ 0&0&2&-1&-1\end {array} \right) =\left(
             \begin{array}{ccc}
               \widetilde{\v}_1 & \cdots & \widetilde{\v}_5 \\
             \end{array}
           \right)
\end{equation*}
Notice that $Q$ is a Gale dual matrix of both $\widetilde{V}$ and the following $F$-matrix
\begin{equation*}
  V=\left( \begin {array}{ccccc} 1&0&5&-2&-3\\ 0&1&3&-3&-2\\ 0&0&6&-3&-3\end {array} \right)=\left(
             \begin{array}{ccc}
               \v_1 & \cdots & \v_5 \\
             \end{array}
           \right)
\end{equation*}
Moreover, it turns out that $\mathcal{T}\cdot V^T=\mathbf{0}$. Then $V$ is a fan matrix of $W$, while $\widetilde{V}$
is a fan matrix of the universal 1-covering $\widetilde{W}$ of $W$. In particular,
\begin{equation*}
  \widetilde{W}= \widehat{W}/H\quad\text{where}\quad H:=\Hom(\Cl(\widetilde{W}),\C^*)\cong\Hom(\Z^2,\C^*)
\end{equation*}
The canonical torsion free 1-covering $\widetilde{X}$ of $X$ is then given by $\widetilde{X}= \widehat{X}/H$. It is a
MDS whose canonical ambient toric variety is given by $\widetilde{W}$. In particular,
$\widetilde{i}:\widetilde{X}\hookrightarrow\widetilde{W}$ is a neat embedding.
Notice that $X$ is quasi-smooth because critical points of the generator
\begin{equation*}
  f=x_1^3x_2^9+x_1^{12}x_4^{18}+x_2^6x_3^6+x_3^{12}x_4^6+x_4^{12}x_5^6
\end{equation*}
are given in Cox coordinates by
\begin{equation*}
  \mathcal{C}=\{x_2 = x_4 = 0\}\cup \{x_1 = x_3= x_4 = 0\}\cup\{ x_1 = x_3 = x_5 = 0\}
\end{equation*}
which is the union of a curve and two points, namely $\mathcal{C}=l\cup \{p\}\cup\{q\}$. The curve $l$ and the point $q$ are contained in the unstable locus $B_X$, so that they do not give any singular point of $X$. The point $p$ is the special point of the maximal cone $\langle\v_1,\v_3,\v_4\rangle$ in the fan of $W$. Since
\begin{equation*}
  \det\left(
        \begin{array}{ccc}
          \v_1 & \v_2 & \v_3 \\
        \end{array}
      \right)=\left|
                      \begin{array}{ccc}
                        1 & 5 & -2 \\
                        0 & 3 & -3 \\
                        0 & 6 & -3 \\
                      \end{array}
      \right|=9
\end{equation*}
it follows that $p$ is a singular point of $W$, too. Then
$$\Sing(X)=\{p\}\subseteq\Sing(W)\cap X\quad\Longrightarrow\quad W_{\text{\rm reg}}\cap X\subseteq X_{\text{\rm reg}}$$
The canonical ambient toric variety $W$ is $\Q$-factorial and already complete, that is $W=Z$, and
$$\Nef(Z)=\left\langle\q_2,\q_3
\right\rangle$$
In particular, $f$ is homogeneous of degree $d_K(f)=[X]=\left(
                                                          \begin{array}{c}
                                                            12 \\
                                                            18 \\
                                                          \end{array}
                                                        \right)=6\,\q_2+6\,\q_3$, which is in the relative interior of $\Nef(X)$. Then $X$ is a $\Q$-ample divisor in $Z$.
The polytope associated with the divisor $X$ is given by
\begin{equation*}
  \Delta_X=\conv\left(
                  \begin{array}{cccccc}
                    -2 & -2 & -2 & -2 & 1 & 10\\
                    -2 & -2 & 4 & 6 & 7 & -2\\
                    1 & 7 & 1 & -3 & -5 & -5\\
                  \end{array}
                \right)=\conv(\pp_1,\ldots,\pp_6)
\end{equation*}
so giving that $X$ is Cartier and then ample. The check that $X$ is actually very ample is done by observing that
\begin{equation*}
  \forall\, i=1,\ldots,6\quad \N(\De_X\cap M -\pp_i)\ \text{is a \emph{saturated semigroup} in $M$}
\end{equation*}
Referring to \cite[Def.~1.3.4]{CLS} for the definition, the check is easily performed by noticing that
 \begin{eqnarray*}
   \De_X\cap M&=&\left\{(-2, -2, 7), (-2, -1, 6), (-1, -2, 6), (-2, -2, 6), (-2, 0, 5), (-1, -1, 5),\right. \\
  &&(-2, -1, 5), (0, -2, 5), (-1, -2, 5), (-2, -2, 5), (-2, 1, 4), (-1, 0, 4),\\
  &&(-2, 0, 4), (0, -1, 4), (-1, -1, 4), (-2, -1, 4), (1, -2, 4), (0, -2, 4),\\
  &&(-1, -2, 4), (-2, -2, 4), (-2, 2, 3), (-1, 1, 3), (-2, 1, 3), (0, 0, 3),\\
  &&(-1, 0, 3), (-2, 0, 3), (1, -1, 3), (0, -1, 3), (-1, -1, 3), (-2, -1, 3),\\
  &&(2, -2, 3), (1, -2, 3), (0, -2, 3), (-1, -2, 3), (-2, -2, 3), (-2, 3, 2),\\
  &&(-1, 2, 2), (-2, 2, 2), (0, 1, 2), (-1, 1, 2), (-2, 1, 2), (1, 0, 2),\\
  &&(0, 0, 2), (-1, 0, 2), (-2, 0, 2), (2, -1, 2), (1, -1, 2), (0, -1, 2),\\
  &&(-1, -1, 2), (-2, -1, 2), (3, -2, 2), (2, -2, 2), (1, -2, 2), (0, -2, 2),\\
  &&(-1, -2, 2), (-2, -2, 2), (-2, 4, 1), (-1, 3, 1), (-2, 3, 1), (0, 2, 1),\\
  &&(-1, 2, 1), (-2, 2, 1), (1, 1, 1), (0, 1, 1), (-1, 1, 1), (-2, 1, 1),\\
  &&(2, 0, 1), (1, 0, 1), (0, 0, 1), (-1, 0, 1), (-2, 0, 1), (3, -1, 1),\\
  &&(2, -1, 1), (1, -1, 1), (0, -1, 1), (-1, -1, 1), (-2, -1, 1), (4, -2, 1),\\
  &&(3, -2, 1), (2, -2, 1), (1, -2, 1), (0, -2, 1), (-1, -2, 1), (-2, -2, 1),\\
  &&(-1, 4, 0), (-2, 4, 0), (0, 3, 0), (-1, 3, 0), (-2, 3, 0), (1, 2, 0),\\
  &&(0, 2, 0), (-1, 2, 0), (-2, 2, 0), (2, 1, 0), (1, 1, 0), (0, 1, 0),\\
  &&(-1, 1, 0), (-2, 1, 0), (3, 0, 0), (2, 0, 0), (1, 0, 0), (0, 0, 0),\\
  &&(-1, 0, 0), (-2, 0, 0), (4, -1, 0), (3, -1, 0), (2, -1, 0), (1, -1, 0),\\
  &&(0, -1, 0), (-1, -1, 0), (5, -2, 0), (4, -2, 0), (3, -2, 0), (2, -2, 0),\\
  &&(1, -2, 0), (0, -2, 0), (-1, 5, -1), (-2, 5, -1), (0, 4, -1), (-1, 4, -1),\\
  &&(-2, 4, -1), (1, 3, -1), (0, 3, -1), (-1, 3, -1), (-2, 3, -1), (2, 2, -1),\\
  &&(1, 2, -1), (0, 2, -1), (-1, 2, -1), (-2, 2, -1), (3, 1, -1), (2, 1, -1),\\
  &&(1, 1, -1), (0, 1, -1), (-1, 1, -1), (4, 0, -1), (3, 0, -1), (2, 0, -1), \\
  &&(1, 0, -1), (0, 0, -1), (5, -1, -1), (4, -1, -1), (3, -1, -1), (2, -1, -1),\\
  &&(1, -1, -1), (6, -2, -1), (5, -2, -1), (4, -2, -1), (3, -2, -1), (2, -2, -1),\\
  &&(0, 5, -2), (-1, 5, -2), (-2, 5, -2), (1, 4, -2), (0, 4, -2), (-1, 4, -2),\\
     \end{eqnarray*}
  \begin{eqnarray*}
  &&(-2, 4, -2), (2, 3, -2), (1, 3, -2), (0, 3, -2), (-1, 3, -2), (3, 2, -2),\\
  &&(2, 2, -2), (1, 2, -2), (0, 2, -2), (4, 1, -2), (3, 1, -2), (2, 1, -2),\\
  &&(1, 1, -2), (5, 0, -2), (4, 0, -2), (3, 0, -2), (2, 0, -2), (6, -1, -2),\\
  &&(5, -1, -2), (4, -1, -2), (3, -1, -2), (7, -2, -2), (6, -2, -2), (5, -2, -2),\\
  &&(4, -2, -2), (0, 6, -3), (-1, 6, -3), (-2, 6, -3), (1, 5, -3), (0, 5, -3),\\
  &&(-1, 5, -3), (2, 4, -3), (1, 4, -3), (0, 4, -3), (3, 3, -3), (2, 3, -3),\\
  &&(1, 3, -3), (4, 2, -3), (3, 2, -3), (2, 2, -3), (5, 1, -3), (4, 1, -3),\\
  &&(3, 1, -3), (6, 0, -3), (5, 0, -3), (4, 0, -3), (7, -1, -3), (6, -1, -3),\\
  &&(5, -1, -3), (8, -2, -3), (7, -2, -3), (6, -2, -3), (1, 6, -4), (0, 6, -4),\\
  &&(2, 5, -4), (1, 5, -4), (3, 4, -4), (2, 4, -4), (4, 3, -4), (3, 3, -4),\\
  &&(5, 2, -4), (4, 2, -4), (6, 1, -4), (5, 1, -4), (7, 0, -4), (6, 0, -4),\\
  &&(8, -1, -4), (7, -1, -4), (9, -2, -4), (8, -2, -4), (1, 7, -5), (2, 6, -5),\\
  &&(3, 5, -5), (4, 4, -5), (5, 3, -5), (6, 2, -5), (7, 1, -5), (8, 0, -5),\\
  &&\left.(9, -1, -5), (10, -2, -5)\right\}
 \end{eqnarray*}
 Then by Theorem~\ref{thm:WeakLefschetz}, $\widetilde{X}\longrightarrow X$ is the universal 1-covering of $X$.

 Let us point out to the reader, that we processed all the data of this example through the Maple package \cite{MDSpackage}.

\subsection{Mori Dream Enriques surfaces}\label{ssez:Enriques}

An Enriques surface is a complex projective smooth surface $X$ with $q(X)=p_g(X)=0$, $2K_X\sim 0$ but $K_X\not\sim 0$.
There are several well known facts about Enriques surfaces, few of them are here recalled:

\begin{proposition}[\S~VIII.15 in \cite{BPVdV}]\label{prop:Enriques}
  Let $X$ be an Enriques surface. Then
  \begin{enumerate}
    \item $\Cl(X)\cong\Z^{10}\oplus\Z/2\Z$, the torsion part being generated by the canonical class $[K_X]$; then
        $X$ has rank $r=10$;
    \item the fundamental group of $X$ is $\pi_1(X)\cong\Z/2\Z$;
    \item if $\widetilde{X}\twoheadrightarrow X$ is the universal covering space of $X$, then $\widetilde{X}$ is a
        \emph{K3 surface}, that is a complex smooth projective surface with $K_X\sim0$ and $q(X)=0$.
  \end{enumerate}
\end{proposition}

Enriques surfaces which are MDS are very special inside the 10-dimensional moduli space of Enriques surfaces. In fact
they correspond to those admitting a finite automorphism group \cite[Thm.~5.1.3.12]{ADHL} and explicitly classified by
Kondo \cite{Kondo}: namely they consist of two 1-dimensional fa\-mi\-lies and five 0-dimensional fa\-mi\-lies (see
also \cite[Thm.~5.1.6.1]{ADHL}).
The following result was firstly conjectured by Dolgachev \cite[Conj.~4.7]{Dolgachev} and then proved by Kondo
\cite[Cor.~6.3]{Kondo}.

\begin{theorem}[Dolgachev-Kondo]\label{thm:D-K}
  Let $X$ be an Enriques surface and $\widetilde{X}$ its K3 universal covering. Then $\Aut(\widetilde{X})$ is
  infinite.
\end{theorem}

Since an Enriques surface $X$ is smooth, the canonical 1-covering of $X$, whose existence is guaranteed by
Theorem~\ref{thm:notors1-cov} when $X$ is a MDS, is actually unramified by Corollary~\ref{cor:ZNpurity}, so giving
precisely the universal topological covering $\phi:\widetilde{X}\twoheadrightarrow X$. Then
Dolgachev-Kondo Theorem \ref{thm:D-K} implies that $\widetilde{X}$ cannot be a MDS, by \cite[Thm.~5.1.3.12]{ADHL}, that is \emph{the
canonical closed embedding $\widetilde{X}\hookrightarrow \widetilde{W}$ cannot be a neat embedding}.

Anyway, Theorem~\ref{thm:notors1-cov} allows us to conclude some interesting properties of the canonical toric
embedding $X\hookrightarrow W$, of a Mori Dream Enriques surface $X$, and its lifting to canonical 1-coverings
$\widetilde{X}\hookrightarrow\widetilde{W}$, summarized as follows:

\begin{corollary}\label{cor:Enriques}
  Let $X$ be a Mori Dream Enriques surface, $i:X\hookrightarrow W$ its canonical toric embedding and consider the
  natural commutative diagram of embeddings and 1-coverings:
\begin{equation*}
          \xymatrix{\widetilde{X}\ar@{->>}^\phi[d]\ar@{^(->}[r]^-{\widetilde{i}}&\widetilde{W}
          \ar@{->>}[d]^-\vf\\
                    X\ar@{^(->}[r]^-i&W}
        \end{equation*}
        Then:
        \begin{enumerate}
          \item the canonical 1-covering $\phi:\widetilde{X}\twoheadrightarrow X$ is the universal (1-)covering of
              $X$,
          \item $\Cl(\widetilde{X})$ and $\Cl(\widetilde{W})$ are free groups,
   \item $\widetilde{r}:=\rk(\Cl(\widetilde{X})>r:=\rk(\Cl(X))=\rk(\Cl(W))=
   \rk(\Cl(\widetilde{W}))=10$,
   \item both $X$ and $W$ have torsion Picard group,
   \item both the  toric ambient varieties $W$ and $\widetilde{W}$ do not admit any fixed point by the torus
       action.
        \end{enumerate}
\end{corollary}

\begin{proof}
(1) follows by the smoothness of $X$ and Zariski-Nagata purity, namely by Corollary~\ref{cor:ZNpurity}.

(2) follows by Theorem~\ref{thm:QUOT+} for what's concerning the universal 1-covering $\widetilde{W}$, while it is a
classically well known fact for what's concerning the universal topological $K3$ covering $\widetilde{X}$.

(3) follows by the Dolgachev-Kondo Theorem~\ref{thm:D-K}, keeping in mind the equivalent conditions (2) and (3) in the
statement of Theorem~\ref{thm:notors1-cov}.

(4) is Proposition~\ref{prop:Enriques}~(1), for what's concerning $X$, and follows by
Proposition~\ref{prop:neat} when recalling that the canonical toric embedding $X\hookrightarrow W$ is neat.

(5) for $W$ is a consequence of item (2). In fact, since $\Pic(W)$ admits a non-trivial torsion subgroup,
the fan $\Si$ of $W$ cannot admit maximal cones of full dimension $\dim(W)$, that is $W$ cannot admit any fixed point
under the torus action. This fact lifts to the universal 1-covering $\widetilde{W}$ by the construction of its fan
$\widetilde{\Si}$ as explained by (\ref{Sigmatilde}) in the proof of Theorem~\ref{thm:QUOT+}.
\end{proof}

\begin{remark}\label{rem:}
Let us emphasize that Corollary~\ref{cor:Enriques} implies that
 the universal $K3$ covering of a Mori Dream Enriques surface (that is an Enriques surface with finite
    automorphism group) admits a canonical embedding as a smooth subvariety of a $\Q$-factorial toric variety,
    whose class group is a free abelian group of rank 10 and whose torus action does not admit any fixed point.
\end{remark}

\begin{acknowledgements}
I would like to thank Lea Terracini, Cinzia Casagrande and Antonio Laface for several fruitful discussions and
suggestions. I also would like to thank J.~Buczy{\'n}sky and W.~Buczy{\'n}ska for inviting me to visit the I.M.P.A.N. of the S.~Banach International Mathematical Centre in Warsaw, where we had very interesting conversations which aided me to improve some of the results here presented.

Last but not least, many thanks go to the Referees for their careful reading of the manuscript and useful suggestions helping me in fixing many inaccuracies.
\end{acknowledgements}

\halfline
\noindent \textbf{Data availability}
The only database used in carrying out the present study is available at \cite{CRdb}, as reported into the References.

\halfline
\noindent \textbf{Competition or conflict of interests}
The Author certifies that no competition or conflict of interests can in any way be connected to the publication of this article.

\bibliography{MILEA}

\newcommand{\noopsort}[1]{}
\begin{thebibliography}{10}

\bibitem{AL}
{\sc Artebani, M., and Laface, A.}
\newblock Hypersurfaces in {M}ori dream spaces.
\newblock {\em J. Algebra 371\/} (2012), 26--37.

\bibitem{ADHL}
{\sc Arzhantsev, I., Derenthal, U., Hausen, J., and Laface, A.}
\newblock {\em Cox rings}, vol.~144 of {\em Cambridge Studies in Advanced
  Mathematics}.
\newblock Cambridge University Press, Cambridge, 2015.

\bibitem{A-McD}
{\sc Atiyah, M.~F., and Macdonald, I.~G.}
\newblock {\em Introduction to commutative algebra}, economy~ed.
\newblock Addison-Wesley Series in Mathematics. Westview Press, Boulder, CO,
  2016.
\newblock For the 1969 original see [MR0242802].

\bibitem{BPVdV}
{\sc Barth, W., Peters, C., and Van~de Ven, A.}
\newblock {\em Compact complex surfaces}, vol.~4 of {\em Ergebnisse der
  Mathematik und ihrer Grenzgebiete (3) [Results in Mathematics and Related
  Areas (3)]}.
\newblock Springer-Verlag, Berlin, 1984.

\bibitem{BC-RGST}
{\sc Bhatt, B., Carvajal-Rojas, J., Graf, P., Schwede, K., and Tucker, K.}
\newblock \'{E}tale fundamental groups of strongly {$F$}-regular schemes.
\newblock {\em Int. Math. Res. Not. IMRN}, 14 (2019), 4325--4339.

\bibitem{BCHMcK}
{\sc Birkar, C., Cascini, P., Hacon, C.~D., and McKernan, J.}
\newblock Existence of minimal models for varieties of log general type.
\newblock {\em J. Amer. Math. Soc. 23}, 2 (2010), 405--468.

\bibitem{Braun}
{\sc Braun, L.}
\newblock The local fundamental group of a kawamata log terminal singuarity is
  finite.
\newblock {\em {\urlfont arXiv:2004.00522}\/}.

\bibitem{Buczynska}
{\sc Buczy\'{n}ska, W.}
\newblock Fake weighted projective spaces.
\newblock {\em {\urlfont arXiv:0805.1211}\/}.

\bibitem{C-RST}
{\sc Carvajal-Rojas, J., Schwede, K., and Tucker, K.}
\newblock Fundamental groups of {$F$}-regular singularities via
  {$F$}-signature.
\newblock {\em Ann. Sci. \'{E}c. Norm. Sup\'{e}r. (4) 51}, 4 (2018), 993--1016.

\bibitem{Catanese}
{\sc Catanese, F.}
\newblock Q.{E}.{D}. for algebraic varieties.
\newblock {\em J. Differential Geom. 77}, 1 (2007), 43--75.

\bibitem{Cox}
{\sc Cox, D.~A.}
\newblock The homogeneous coordinate ring of a toric variety.
\newblock {\em J. Algebraic Geom. 4}, 1 (1995), 17--50.

\bibitem{CLS}
{\sc Cox, D.~A., Little, J.~B., and Schenck, H.~K.}
\newblock {\em Toric varieties}, vol.~124 of {\em Graduate Studies in
  Mathematics}.
\newblock American Mathematical Society, Providence, RI, 2011.

\bibitem{Danilov}
{\sc Danilov, V.~I.}
\newblock The geometry of toric varieties.
\newblock {\em Uspekhi Mat. Nauk 33}, 2(200) (1978), 85--134.

\bibitem{Dolgachev}
{\sc Dolgachev, I.}
\newblock On automorphisms of {E}nriques surfaces.
\newblock {\em Invent. Math. 76}, 1 (1984), 163--177.

\bibitem{Ewald96}
{\sc Ewald, G.}
\newblock {\em Combinatorial convexity and algebraic geometry}, vol.~168 of
  {\em Graduate Texts in Mathematics}.
\newblock Springer-Verlag, New York, 1996.

\bibitem{ES}
{\sc Ewald, G., and Ishida, M.-N.}
\newblock Completion of real fans and {Z}ariski-{R}iemann spaces.
\newblock {\em Tohoku Math. J. (2) 58}, 2 (2006), 189--218.

\bibitem{GOST}
{\sc Gongyo, Y., Okawa, S., Sannai, A., and Takagi, S.}
\newblock Characterization of varieties of {F}ano type via singularities of
  {C}ox rings.
\newblock {\em J. Algebraic Geom. 24}, 1 (2015), 159--182.

\bibitem{G-MacPh}
{\sc Goresky, M., and MacPherson, R.}
\newblock {\em Stratified {M}orse theory}, vol.~14 of {\em Ergebnisse der
  Mathematik und ihrer Grenzgebiete (3) [Results in Mathematics and Related
  Areas (3)]}.
\newblock Springer-Verlag, Berlin, 1988.

\bibitem{GR}
{\sc Grauert, H., and Remmert, R.}
\newblock Komplexe {R}\"aume.
\newblock {\em Math. Ann. 136\/} (1958), 245--318.

\bibitem{GKP}
{\sc Greb, D., Kebekus, S., and Peternell, T.}
\newblock \'{E}tale fundamental groups of {K}awamata log terminal spaces, flat
  sheaves, and quotients of abelian varieties.
\newblock {\em Duke Math. J. 165}, 10 (2016), 1965--2004.

\bibitem{SGA2}
{\sc Grothendieck, A.}
\newblock {\em Cohomologie locale des faisceaux coh\'erents et th\'eor\`emes de
  {L}efschetz locaux et globaux ({SGA} 2)}, vol.~4 of {\em Documents
  Math\'ematiques (Paris) [Mathematical Documents (Paris)]}.
\newblock Soci\'et\'e Math\'ematique de France, Paris, 2005.

\bibitem{SGA1}
{\sc Grothendieck, A., and Raynaud, M.}
\newblock {\em Rev\^etements \'etales et groupe fondamental (SGA 1)}.
\newblock S\'eminaire de G\'eom\'etrie Alg\'ebrique du Bois Marie 1960/61.
  \texttt{ arXiv:math/0206203}.

\bibitem{Hamm-Le}
{\sc Hamm, H.~A., and ung Tr\'ang, L.~D.}
\newblock Lefschetz theorems on quasiprojective varieties.
\newblock {\em Bull. Soc. Math. France 113}, 2 (1985), 123--142.

\bibitem{Hartshorne}
{\sc Hartshorne, R.}
\newblock {\em Algebraic geometry}.
\newblock Springer-Verlag, New York-Heidelberg, 1977.
\newblock Graduate Texts in Mathematics, No. 52.

\bibitem{CRdb}
{\sc Hausen, J., and Keicher, S.}
\newblock {\em Cox ring database}.
\newblock available at
  \texttt{http://www.math.uni-tuebingen.de/user/keicher/coxringdb/}.

\bibitem{MDSpackage}
{\sc Hausen, J., and Keicher, S.}
\newblock A software package for {M}ori dream spaces.
\newblock {\em LMS J. Comput. Math. 18}, 1 (2015), 647--659.

\bibitem{Hu-Keel}
{\sc Hu, Y., and Keel, S.}
\newblock Mori dream spaces and {GIT}.
\newblock {\em Michigan Math. J. 48\/} (2000), 331--348.

\bibitem{Jow}
{\sc Jow, S.-Y.}
\newblock A {L}efschetz hyperplane theorem for {M}ori dream spaces.
\newblock {\em Math. Z. 268}, 1-2 (2011), 197--209.

\bibitem{Kollar}
{\sc Koll\`{a}r, J.}
\newblock New examples of terminal and log canonical singularities.
\newblock {\em {\urlfont arXiv:1107.2864}\/}.

\bibitem{Kondo}
{\sc Kond\=o, S.}
\newblock Enriques surfaces with finite automorphism groups.
\newblock {\em Japan. J. Math. (N.S.) 12}, 2 (1986), 191--282.

\bibitem{Lang}
{\sc Lang, S.}
\newblock {\em Algebra}, third~ed., vol.~211 of {\em Graduate Texts in
  Mathematics}.
\newblock Springer-Verlag, New York, 2002.

\bibitem{Milne}
{\sc Milne, J.~S.}
\newblock {\em Lectures on \'{E}tale cohomology}.
\newblock \texttt{http://www.jmilne.org/math/CourseNotes/lec.html}, 2013.

\bibitem{Murre}
{\sc Murre, J.~P.}
\newblock {\em Lectures on an introduction to {G}rothendieck's theory of the
  fundamental group}.
\newblock Tata Institute of Fundamental Research, Bombay, 1967.
\newblock Notes by S. Anantharaman, Tata Institute of Fundamental Research
  Lectures on Mathematics, No 40.

\bibitem{Nagata}
{\sc Nagata, M.}
\newblock Imbedding of an abstract variety in a complete variety.
\newblock {\em J. Math. Kyoto Univ. 2\/} (1962), 1--10.

\bibitem{RS}
{\sc Ravindra, G.~V., and Srinivas, V.}
\newblock The {G}rothendieck-{L}efschetz theorem for normal projective
  varieties.
\newblock {\em J. Algebraic Geom. 15}, 3 (2006), 563--590.

\bibitem{Rohrer11}
{\sc Rohrer, F.}
\newblock Completions of fans.
\newblock {\em J. Geom. 100}, 1-2 (2011), 147--169.

\bibitem{R-wMDS}
{\sc Rossi, M.}
\newblock Embedding non-projective {M}ori dream space.
\newblock {\em Geom. Dedicata 207\/} (2020), 355--393.

\bibitem{RT-LA&GD}
{\sc Rossi, M., and Terracini, L.}
\newblock $\mathbb{Z}$--linear gale duality and poly weighted spaces {(PWS)}.
\newblock {\em Linear Algebra Appl. 495\/} (2016), 256--288.

\bibitem{RT-QUOT}
{\sc Rossi, M., and Terracini, L.}
\newblock A $\mathbb{Q}$--factorial complete toric variety is a quotient of a
  poly weighted space.
\newblock {\em Ann. Mat. Pur. Appl. 196\/} (2017), 325--347.

\bibitem{RT-Pic}
{\sc Rossi, M., and Terracini, L.}
\newblock Embedding the picard group inside the class group: the case of
  $\mathbb{Q}$--factorial complete toric varieties.
\newblock {\em J. Algebraic Combin. 53\/} (2021), 553--573.

\bibitem{Stibitz}
{\sc Stibitz, C.}
\newblock \'{E}tale covers and local algebraic fundamental groups.
\newblock {\em {\urlfont arXiv:1707.08611}\/}.

\bibitem{Sumihiro1}
{\sc Sumihiro, H.}
\newblock Equivariant completion.
\newblock {\em J. Math. Kyoto Univ. 14\/} (1974), 1--28.

\bibitem{Sumihiro2}
{\sc Sumihiro, H.}
\newblock Equivariant completion. {II}.
\newblock {\em J. Math. Kyoto Univ. 15}, 3 (1975), 573--605.

\bibitem{Szamuely}
{\sc Szamuely, T.}
\newblock {\em Galois groups and fundamental groups}, vol.~117 of {\em
  Cambridge Studies in Advanced Mathematics}.
\newblock Cambridge University Press, Cambridge, 2009.

\bibitem{Tian-Xu}
{\sc Tian, Z., and Xu, C.}
\newblock Finiteness of fundamental groups.
\newblock {\em Compos. Math. 153}, 2 (2017), 257--273.

\bibitem{Xu}
{\sc Xu, C.}
\newblock Finiteness of algebraic fundamental groups.
\newblock {\em Compos. Math. 150}, 3 (2014), 409--414.

\end{thebibliography}
\bibliographystyle{acm}

\end{document}